\documentclass[11pt]{amsart}
\usepackage[T1]{fontenc}
\usepackage[utf8]{inputenc}
\usepackage{amsmath,amssymb,amsthm,mathtools,mathrsfs}
\usepackage{xcolor}
\definecolor{linkblue}{RGB}{38,68,105}
\definecolor{citegreen}{RGB}{45,88,68}
\usepackage[colorlinks=true,linkcolor=linkblue,citecolor=citegreen,urlcolor=linkblue,pdfborder={0 0 0},linktoc=all,bookmarksdepth=1]{hyperref}
\theoremstyle{plain}
\newtheorem{thm}{Theorem}[section]
\newtheorem{lem}[thm]{Lemma}
\newtheorem{prop}[thm]{Proposition}
\newtheorem{cor}[thm]{Corollary}
\newtheorem{alphthm}{Theorem}

\theoremstyle{definition}
\newtheorem{defn}[thm]{Definition}
\theoremstyle{remark}
\newtheorem{rem}[thm]{Remark}
\newcommand{\C}{\mathbb C}
\newcommand{\R}{\mathbb R}
\newcommand{\cO}{\mathcal O}
\newcommand{\cT}{\mathcal T_g}
\newcommand{\cM}{\mathcal M}
\newcommand{\cG}{\mathcal G}
\newcommand{\cH}{\mathcal H}
\newcommand{\fk}{\mathfrak k}
\newcommand{\fh}{\mathfrak h}
\newcommand{\Id}{\operatorname{Id}}
\DeclareMathOperator{\End}{End}
\DeclareMathOperator{\Aut}{Aut}
\DeclareMathOperator{\Hom}{Hom}
\DeclareMathOperator{\Lie}{Lie}
\DeclareMathOperator{\Spec}{Spec}
\DeclareMathOperator{\Stab}{Stab}
\DeclareMathOperator{\Comp}{Comp}
\DeclareMathOperator{\im}{im}
\DeclareMathOperator{\tr}{tr}
\DeclareMathOperator{\Ad}{Ad}
\DeclareMathOperator{\pr}{pr}
\DeclareMathOperator{\Diff}{Diff}
\DeclareMathOperator{\SL}{SL}
\DeclareMathOperator{\GL}{GL}
\DeclareMathOperator{\rank}{rank}
\DeclareMathOperator{\id}{id}
\newcommand{\red}{\mathrm{red}}
\newcommand{\an}{\mathrm{an}}
\title[Stratification of universal Higgs moduli]{Stratification of the universal moduli space of $\SL_r(\C)$-Higgs bundles}
\author{Shiyu Cao}
\address{Chern Institute of Mathematics, Nankai University}
\email{shiyucao@126.com}
\date{}

\begin{document}
\begin{abstract}
We study the analytic moduli space of marked polystable $\SL_r(\C)$-Higgs bundles over Teichm\"uller space. The connected components of the orbit-type loci defined by stabilizers of determinant one form a complex Whitney stratification, and the projection to Teichm\"uller space is a holomorphic submersion on every stratum. Every stratum admits local harmonic lifts that depend smoothly on all parameters, including the curve parameter, and the induced vertical hyperk\"ahler tensors are smooth. Every stratum carries a canonical relative holomorphic symplectic form. These forms determine a vertical holomorphic Poisson bracket on the reduced structure sheaf, with symplectic leaves given by the connected orbit-type strata in each fiber. On every stratum of a fiber, the Hitchin--Kobayashi correspondence is biholomorphic and preserves the hyperk\"ahler structure.
\end{abstract}
\maketitle
\tableofcontents

\section{Introduction}\label{sec:introduction}

A Higgs bundle on a closed Riemann surface $X$ is a holomorphic bundle $E$ together with a holomorphic map $\varphi\colon E\to E\otimes K_X$, where $K_X$ is the canonical line bundle. Its rank and degree are those of $E$. Hitchin introduced rank-two Higgs bundles in his study of the self-duality equations on a compact Riemann surface~\cite{Hitchin1987}.

The moduli problem is governed by stability with respect to the Higgs field. The slope of a nonzero bundle is its degree divided by its rank. A Higgs bundle is stable if every nonzero proper $\varphi$-invariant holomorphic subbundle has slope strictly smaller than that of $E$. It is semistable if the slope of each such subbundle is no greater than that of $E$. It is polystable if it is a direct sum of stable Higgs bundles of the same slope. Two semistable objects are $S$-equivalent when their associated graded polystable objects are isomorphic. Thus an $S$-equivalence class has a polystable representative unique up to isomorphism~\cite[Proposition~4.1 and Definition~4.2]{Nitsure1991}.

For a fixed smooth projective curve, Nitsure constructs a quasi-projective coarse moduli scheme of semistable Higgs bundles of fixed rank and degree, whose points are their $S$-equivalence classes~\cite{Nitsure1991}. This construction is also a special case of Simpson's moduli theory for semistable $\Lambda$-modules~\cite[Theorem~4.7]{Simpson1994I}.

Harmonic metrics relate this holomorphic moduli problem to gauge theory. For a smooth projective complex curve, polystable Higgs bundles of degree zero correspond through harmonic metrics to semisimple local systems~\cite[Section~1, Theorem~1]{Simpson1992}. On a curve, the relevant vanishing Chern class condition reduces to $\deg E=0$. For a fixed closed Riemann surface, the Kuranishi slice method constructs the Higgs moduli space as a normal complex space~\cite{FanConstruction}. Analytic methods also establish its quasi-projectivity for a fixed curve~\cite{FanQuasiprojectivity}.

At a polystable point, automorphisms determine the local quotient singularities. The corresponding orbit type records the conjugacy class of the stabilizer. For a fixed curve, the orbit-type decomposition has smooth hyperk\"ahler strata, a complex Whitney stratification, and a compatible holomorphic Poisson bracket~\cite[Theorems~A--D, arXiv v1]{FanOrbitTypes}. The finite-dimensional reduction results~\cite[Theorem~1.4 and Sections~4.2--4.9, arXiv v2]{Mayrand2020Stratification} provide orbit-type stratifications and symplectic local models for singular hyperk\"ahler quotients.

We allow the marked curve and the Higgs bundle to vary simultaneously. A marking is an orientation-preserving diffeomorphism from a fixed closed connected surface $\Sigma$ of genus $g\geq2$, considered up to isotopy. The Teichm\"uller space $\cT$ of $\Sigma$ parametrizes the marked curves $X_b$. Collier--Toulisse--Wentworth construct the stable joint moduli space over Teichm\"uller space~\cite[Theorem~4.6 and Section~4.8]{CTW25}. For the polystable problem, we use the analytic space and its relative Kuranishi quotient families constructed in~\cite{CaoUniversalPolystable}.\footnote{Numbered references to~\cite{CaoUniversalPolystable} follow the revised companion manuscript of 3 December 2026.}

\'Alvarez-C\'onsul--Garcia-Fernandez--Garc\'ia-Prada--Trautwein define the universal quotient of polystable Higgs bundles and obtain K\"ahler fibrations and pseudo-K\"ahler structures on suitable smooth open subsets through coupled equations~\cite[Section~5.1 and Theorems~5.18, 5.22]{UniversalHitchinModuliSpaces}. Hitchin constructs the varying complex structure on the smooth universal moduli space from harmonic-map energy and a symplectic Ehresmann connection, and obtains its vertical holomorphic Poisson tensor~\cite[Sections~2.1 and~5]{HitchinUniversalHiggsModuliSpace}. We study the orbit-type strata throughout the polystable space, including the singularities represented by strictly polystable bundles. The relative quotient models identify the strata locally as products with the curve parameter. Smooth harmonic lifts then identify their vertical tangent bundles and tensors, while an exact trace identity allows the relative symplectic forms and Poisson brackets to agree on overlapping charts.

We write $p\colon\cM\to\cT$ for this universal moduli space. We consider marked $\SL_r(\C)$-Higgs bundles of rank $r\geq2$, with a specified trivialization $\tau$ of $\det E$, a trace-free Higgs field, and isomorphisms preserving $\tau$. The analytic framework, including Sobolev regularity $k\geq4$ and the gauge groups of determinant one, is given in Section~\ref{sec:standing}. Definition~\ref{defn:types} specifies the partition $\mathcal P_{\mathrm{conn}}$ into connected components of the orbit-type loci. For a stratum $\mathscr S$, we write $T_{\mathscr S/\cT}$ for its vertical tangent bundle. We study how these strata, their symplectic forms, and their harmonic lifts vary with the marked curves. The arguments compare the gauge stabilizers with the finite-dimensional stabilizers in the relative Kuranishi models.

\subsection{Strata and harmonic lifts}
\begin{alphthm}\label{thm:main-strata}
The partition $\mathcal P_{\mathrm{conn}}$ is a locally finite complex Whitney stratification of $\cM$. The projection $p|_{\mathscr S}$ is a holomorphic submersion on every stratum $\mathscr S$.
\end{alphthm}

The product charts also control limits of vertical tangent spaces. Corollary~\ref{cor:thom} records this compatibility as Thom's condition~$(a_p)$.

A local harmonic lift is a family of solutions to the Hitchin equations representing the corresponding moduli classes. Such lifts relate the vertical tangent spaces of the strata to the fixed-curve gauge quotients. For $b\in\cT$, we denote the fiberwise Hitchin--Kobayashi homeomorphism by $\mathrm{HK}_b$.

\begin{alphthm}\label{thm:main-harmonic}
Every stratum $\mathscr S$ has smooth local harmonic lifts obtained by gauge transformations of determinant one. The harmonic tangent spaces invariant under the stabilizer form a smooth complex vector bundle canonically isomorphic to $T_{\mathscr S/\cT}$. The $L^2$ metric and the three complex structures descend to canonical smooth tensors on this bundle over $\mathscr S$.

For every $b\in\cT$, $\mathrm{HK}_b$ preserves the orbit types determined by the stabilizers and is biholomorphic on every orbit-type stratum. Under $\mathrm{HK}_b$, the quotient hyperk\"ahler structures agree with the tensors above on the connected components of $\mathscr S\cap p^{-1}(b)$.
\end{alphthm}

The comparison with the fixed-curve results~\cite[Theorems~A and~C]{FanOrbitTypes} uses conjugacy in the gauge groups of determinant one and descent under unitary gauge isomorphisms that preserve the determinant trivialization, including permutations of stable isotypic blocks. Sections~\ref{sec:local} and~\ref{sec:harmonic} establish these comparisons; Section~\ref{sec:whitney} proves the global Whitney assertion. The invariant deformation complex also gives the dimension of every stratum in terms of the ranks of its stable factors; see Corollary~\ref{cor:stratum-dimension} and the rank-two example following it. Proposition~\ref{prop:gauge-preimages} proves that the preimages of strata of $\cM$ under the gauge quotient maps are embedded Banach submanifolds in relative Sobolev configuration spaces and gives local product charts for their quotient projections.

\subsection{Relative forms and the Poisson bracket}
The symplectic geometry of a stratum $\mathscr S\subset\cM$ is relative to its map to Teichm\"uller space. We use the vertical trace pairing~\eqref{eq:trace}, with deformation cocycles defined in Section~\ref{sec:deformation}.

\begin{alphthm}\label{thm:main-forms}
Every stratum $\mathscr S$ has a canonical form
\begin{equation*}
 \Omega^{\tr}_{\mathscr S/\cT}
 \in H^0(\mathscr S,\Omega^2_{\mathscr S/\cT}),\qquad
 d_{\mathscr S/\cT}\Omega^{\tr}_{\mathscr S/\cT}=0,
\end{equation*}
nondegenerate on $T_{\mathscr S/\cT}$. It is evaluated by the trace pairing on vertical deformation cocycles and agrees with the invariant harmonic pairing. In relative Kuranishi quotient charts, these forms agree with the reduced trace forms on the local orbit-type strata.
\end{alphthm}

The comparison on all local strata extends the fixed-curve comparison~\cite[Lemmas~7.2--7.3 and Theorem~7.1]{FanOrbitTypes}. This compatibility permits the local brackets to glue on the reduced analytic structure sheaf.

\begin{alphthm}\label{thm:main-poisson}
There is a unique holomorphic Poisson bracket on $\cO_{\cM}$, bilinear over $p^{-1}\cO_{\cT}$, whose restriction to each connected component $\mathscr S_b$ of $\mathscr S\cap p^{-1}(b)$ is $\bigl(\Omega^{\tr}_{\mathscr S/\cT}|_{\mathscr S_b}\bigr)^{-1}$, for every stratum $\mathscr S$ and $b\in\cT$.
\end{alphthm}

In particular, pullbacks of base functions are Casimirs: $\{p^{\#}f,g\}=0$ for local holomorphic functions $f$ on $\cT$ and $g$ on $\cM$.

The symplectic leaves are the connected components of $\mathscr S\cap p^{-1}(b)$, by Corollary~\ref{cor:leaves}. The hyperk\"ahler assertions of Theorem~\ref{thm:main-harmonic} concern these fixed-curve strata; the forms of Theorem~\ref{thm:main-forms} are holomorphic relative forms on the strata of $\cM$.

\subsection*{Organization of the paper}
Section~\ref{sec:preliminaries} recalls the analytic setting, the deformation complex, and relative Kuranishi families. Section~\ref{sec:local} compares gauge stabilizers with stabilizers in the Kuranishi models and describes the strata using fixed subspaces. Section~\ref{sec:whitney} proves local finiteness, the frontier property, and the Whitney conditions. Section~\ref{sec:harmonic} constructs local harmonic lifts, compares them on overlaps, and describes the gauge preimages of the strata. Section~\ref{sec:forms} shows that the relative Kuranishi maps preserve the trace pairing and descends this pairing to the strata. Section~\ref{sec:poisson} glues the local Poisson brackets into a global bracket whose symplectic leaves are the connected orbit-type strata in each fiber. It also studies changes of marking and Higgs-field scaling.

\subsection*{Acknowledgments}
The author thanks Professor Qiongling Li for her support and Professors Kefeng Liu and Huitao Feng for their guidance over the years.

\subsection*{Use of AI tools}
The mathematical idea behind this paper belongs to the author. During the preparation, the author used ChatGPT 5.6 sol and 6 astra for literature searches, editorial revision, and checks of mathematical arguments and notation. The author takes full responsibility for the content.

\section{Preliminaries}\label{sec:preliminaries}
We collect the marked objects, gauge groups, and the deformation complex used throughout the paper. The relative Kuranishi families recalled below supply the charts used to compare orbit types and trace pairings.

\subsection{Higgs bundles and gauge groups}\label{sec:standing}
Let $\Sigma$ be a closed connected oriented smooth surface of genus $g\geq2$, and let $r\geq2$ and $k\geq4$ be integers. We write $\cT$ for its Teichm\"uller space, and $X_b$ for the curve corresponding to $b\in\cT$. A marked $\SL_r(\C)$-Higgs bundle is $\mathbf x=(X,E,\tau,\varphi)$, where $X$ is equipped with a marking by $\Sigma$, $E$ is a rank-$r$ holomorphic bundle, $\tau\colon\det E\xrightarrow{\sim}\cO_X$ is specified, and $\varphi\in H^0(X,\End_0E\otimes K_X)$. All isomorphisms preserve the marking, $\tau$, and the Higgs field. The terms universal and joint both refer to simultaneous variation of $X$ and $(E,\tau,\varphi)$.

Let $\cM$ be the reduced normal Hausdorff analytic space of marked polystable objects of~\cite[Theorems~A--B]{CaoUniversalPolystable}, with projection $p\colon\cM\to\cT$. On a fixed smooth Hermitian reference bundle with a specified determinant trivialization, let $\cG_{k+1}$ and $\cG_{k+1}^{\C}$ be the unitary and complex $W^{k+1,2}$ gauge groups of determinant one. Their smooth versions are $\cG$ and $\cG^\C$. The configuration space has $W^{k,2}$ coefficients. For $c=(\bar\partial_E,\varphi)$ and $g\in\cG_{k+1}^{\C}$, we use the right action
\begin{equation*}
 c^g=(g^{-1}\bar\partial_Eg,g^{-1}\varphi g),
 \qquad (c^g)^u=c^{gu}.
\end{equation*}
We call $(\bar\partial_E,\varphi)$ a Higgs pair. With the Hermitian metric fixed, it is equivalently represented by a pair $(\nabla,\varphi)$ consisting of the Chern connection and the Higgs field. Under a complex gauge transformation, the transformed Dolbeault operator determines the transformed Chern connection. All gauge groups and stabilizers include every connected component.

Here $W^{s,2}$ denotes sections with square-integrable covariant derivatives through order $s$, using smooth background metrics and a reference connection on the closed surface. Different background choices give equivalent norms. In a local family these norms are uniformly equivalent on relatively compact parameter neighborhoods. The harmonic metric chosen at a base point has determinant one with respect to $\tau$. The notation $\cM$ and $p$ agrees with~\cite{CaoUniversalPolystable}. We denote the surface complex structure by $j_b$ and reserve $J_b$ for one of the three complex structures on the vertical tangent bundle.

\begin{defn}\label{defn:types}
An orbit-type locus in $\cM$ consists of points for which the stabilizers of polystable representatives are conjugate as embedded subgroups of $\cG_{k+1}^{\C}$. Let $\mathcal P_{\mathrm{conn}}$ be the partition into connected components of these loci, with their subspace topologies.
\end{defn}

For a partition $\mathcal P$, write $\Comp(\mathcal P)$ for its connected refinement. For a smooth orbit-type locus $\mathscr S$ on which $p$ is a submersion, put $T_{\mathscr S/\cT}=\ker d(p|_{\mathscr S})$.

A complex Whitney stratification is a locally finite partition into connected, locally closed complex submanifolds satisfying the frontier property and Whitney conditions \textup{(a)} and \textup{(b)} in local analytic embeddings~\cite[Definition~2.7]{Mayrand2020Stratification}. The frontier property means that $\mathscr S\cap\overline{\mathscr R}\ne\varnothing$ implies $\mathscr S\subseteq\overline{\mathscr R}$ for any two strata.

\subsection{Deformation complex}\label{sec:deformation}
For the Higgs deformation complex, we follow~\cite[Definition~2.2 and Lemma~2.15]{CaoUniversalPolystable} and~\cite[Section~2]{FanOrbitTypes}. For reduced analytic quotients we use~\cite[Section~2.4]{Mayrand2020Stratification}. Fix $x=[\mathbf x]\in\cM$, where $\mathbf x=(X,E,\tau,\varphi)$, and choose a solution $\widehat{\mathbf x}=(X,\nabla_h,\tau,\varphi)$ to the Hitchin equations whose determinant connection is trivial under $\tau$. The Higgs deformation complex is
\begin{equation*}
 C_x^\bullet=
 [\End_0E\xrightarrow{[\varphi,\cdot]}
 \End_0E\otimes K_X].
\end{equation*}
Its Dolbeault resolution has arrows
\begin{equation*}
 D_{0,x}\xi=(\bar\partial_{\End}\xi,[\varphi,\xi]),\qquad
 D_{1,x}(a,\psi)=\bar\partial_{\End}\psi+[\varphi,a].
\end{equation*}
For total degrees $i,j$, the bracket is $[s,t]=s\wedge t-(-1)^{ij}t\wedge s$. The harmonic metric $h$ and the K\"ahler metric on $X$ define the adjoints,
\begin{equation*}
 \Delta_{j,x}=D_{j-1,x}D_{j-1,x}^*+D_{j,x}^*D_{j,x},
 \qquad \cH_x^j=\ker\Delta_{j,x},
\end{equation*}
with absent arrows omitted. We denote harmonic projection by $\mathsf H_{j,x}$ and the Green operator by $G_{j,x}$, with $G_{j,x}=0$ on $\cH_x^j$. We set $V_x=\mathbb H^1(C_x^\bullet)$ and use the harmonic isomorphism $V_x\cong\cH_x^1$ to regard elements of $V_x$ as harmonic representatives.

Let $H_x=\Aut(E,\tau,\varphi)$ and $K_x=H_x\cap\cG$. The linear action is the left action $h\cdot v=hvh^{-1}$. Our vertical trace pairing and moment-map normalization are
\begin{equation}\label{eq:trace}
 \Omega_{x,\C}((a,\psi),(a',\psi'))
   =\int_X\tr(a\wedge\psi'-a'\wedge\psi),
\end{equation}
\begin{equation}\label{eq:moment}
\begin{aligned}
 \langle\nu_{x,\C}(v),\xi\rangle
 &=\tfrac12\Omega_{x,\C}([\xi,v],v)
 =\int_X\tr(\xi[a,\psi]),\\
 d\langle\nu_{x,\C},\xi\rangle
 &=\iota_{[\xi,v]}\Omega_{x,\C}.
\end{aligned}
\end{equation}
Integration uses the orientation of $\Sigma$. Serre duality is realized by the trace pairing
\begin{equation*}
 \operatorname{SD}_x(z)(\xi)=\int_X\tr(\xi z),
\end{equation*}
and gives an isomorphism $\operatorname{SD}_x\colon\cH_x^2\xrightarrow{\sim}(\cH_x^0)^*=\fh_x^*$, where $\fh_x=\Lie H_x$. Thus
\begin{equation*}
 \nu_{x,\C}(v)
 =\operatorname{SD}_x(\tfrac12\mathsf H_{2,x}[v,v]).
\end{equation*}
We define the quadratic model by
\begin{equation*}
 Q_x=(\nu_{x,\C}^{-1}(0))_{\red},\qquad
 Y_x=(\Spec\C[Q_x]^{H_x})^{\an},\qquad q_x\colon Q_x\to Y_x.
\end{equation*}
For $y\in Y_x$, let $H_x\cdot v\subset q_x^{-1}(y)$ be the unique closed orbit. The orbit type of $y$ is the $H_x$-conjugacy class of $(H_x)_v$.

\subsection{Stabilizers}\label{sec:stabilizers}
At a solution to the Hitchin equations, the algebra of Higgs endomorphisms is a finite-dimensional algebra of $\nabla_h$-parallel sections closed under the $h$-adjoint~\cite[Lemma~2.7]{CaoUniversalPolystable}. An orthogonal decomposition into isotypic summands has the form
\begin{equation*}
 (E,\varphi)=\bigoplus_{\alpha=1}^s
       (E_\alpha,\varphi_\alpha)\otimes V_\alpha,
\end{equation*}
with pairwise nonisomorphic stable factors, ranks $d_\alpha=\rank E_\alpha$, and fixed multiplicity spaces $V_\alpha=\C^{m_\alpha}$. The endomorphism algebra is
\begin{equation*}
 \mathcal A_x=\bigoplus_{\alpha=1}^s
      \Id_{E_\alpha}\otimes\End(V_\alpha).
\end{equation*}
The units of determinant one of $\mathcal A_x$ form the reductive group
\begin{equation*}
\begin{aligned}
 H_x&=\ker\left(
 \prod_\alpha\GL(V_\alpha)\longrightarrow\C^*,\quad
 (a_\alpha)\longmapsto\prod_\alpha\det(a_\alpha)^{d_\alpha}\right)\\
 &=K_x\exp(i\fk_x),\qquad \fk_x=\Lie K_x.
\end{aligned}
\end{equation*}
The polar decomposition includes all components~\cite[Section~2.4]{CaoUniversalPolystable}. Conjugation by $H_x$ commutes with the differentials $D_{j,x}$, their adjoints, the harmonic projections, and the Green operators. These properties follow from~\cite[Lemmas~2.7 and~2.15]{CaoUniversalPolystable}.

\subsection{Relative Kuranishi charts and families}
\label{sec:relative-families}
The relative Kuranishi construction provides a connected Beltrami coordinate ball $B_x\subset\cT$, a $K_x$-invariant open ball $U_x\subset V_x$, and the holomorphic map of~\cite[Proposition~2.10 and Lemma~2.11]{CaoUniversalPolystable},
\begin{equation*}
 A_x(b,v)=v+w_x(b,v),\qquad
 w_x(b,v)\in\im D_{1,x}^*.
\end{equation*}
On $B_x\times U_x$, the map $A_x$ is holomorphic with values in $W^{m,2}$ for every integer $m\geq0$. Its coefficients are smooth on $\Sigma$, and it satisfies
\begin{equation*}
 D_{0,x}^*A_x(b,v)=0,
 \qquad \mathsf H_{1,x}A_x(b,v)=v.
\end{equation*}
The quadratic identity in~\cite[Theorem~2.16]{CaoUniversalPolystable} gives the reduced locus of solutions
\begin{equation*}
 Z_x=B_x\times(Q_x\cap U_x).
\end{equation*}
The map $A_x$ is $H_x$-equivariant wherever both sides are defined.

Let $\theta_x$ denote the holomorphic Higgs family on $Z_x$. By~\cite[Theorem~2.19 and Lemma~2.20]{CaoUniversalPolystable}, there is a sufficiently small Stein neighborhood $\mathcal W_x\subset Y_x$ of $q_x(0)$ such that
\begin{equation}\label{eq:saturated-domain}
 \widetilde U_x:=q_x^{-1}(\mathcal W_x)
       \subset H_x\cdot(Q_x\cap U_x).
\end{equation}
The family extends equivariantly to $B_x\times H_x\cdot(Q_x\cap U_x)$. We denote its restriction to $B_x\times\widetilde U_x$ by $\widetilde\theta_x$; it agrees with $\theta_x$ on their common domain. The right gauge action and the left model action are related by
\begin{equation}\label{eq:relative-equivariance}
 \widetilde\theta_x(b,h\cdot v)
     =\widetilde\theta_x(b,v)^{h^{-1}}
 \qquad(h\in H_x).
\end{equation}
The same identity holds for $\theta_x$ whenever both parameters lie in $Z_x$. There is a biholomorphic chart
\begin{equation*}
 \chi_x\colon B_x\times\mathcal W_x
 \xrightarrow{\sim}\mathscr U_x\subset\cM,
 \qquad p\chi_x(b,y)=b.
\end{equation*}
The $S$-equivalence class of $\widetilde\theta_x(b,v)$ is $\chi_x(b,q_x(v))$, and~\cite[Corollary~2.21]{CaoUniversalPolystable} gives
\begin{equation}\label{eq:stabilizer-comparison}
 \Stab_{\cG_{k+1}^{\C}}\widetilde\theta_x(b,v)=(H_x)_v.
\end{equation}
This identity also holds on $Z_x$. The family $\widetilde\theta_x(b,v)$ is polystable exactly when $H_x\cdot v$ is closed, by~\cite[Proposition~3.7]{CaoUniversalPolystable}.

Let $T$ be a reduced analytic germ with a distinguished point, and let a holomorphic Higgs family over $T$ be equipped with a specified isomorphism from its distinguished fiber to $\mathbf x$. After shrinking $T$, there is a holomorphic map $f:T\to Z_x$ over $B_x$, together with an isomorphism from the given family to $f^*\theta_x$ that preserves the determinant trivialization. The family isomorphism restricts to the prescribed isomorphism at the distinguished point and is holomorphic on the reduced total space of the curve family. Local completeness and regularity follow from~\cite[Proposition~2.10, Lemmas~2.11--2.13, and Theorem~2.14]{CaoUniversalPolystable}.

\subsection{Topology and regularity}\label{sec:topology}
For $b\in\cT$, let $\cM_{\mathrm{Hit},b}$ denote the quotient by $\cG$ of smooth solutions to the Hitchin equations on $X_b$ whose determinant connection is trivial under $\tau$ and whose Higgs field is trace-free. Put $\cM_b=p^{-1}(b)$. The map that sends a solution to its underlying polystable Higgs bundle defines a homeomorphism
\begin{equation*}
 \mathrm{HK}_b\colon\cM_{\mathrm{Hit},b}
      \xrightarrow{\ \sim\ }\cM_b.
\end{equation*}
Solutions to the Hitchin equations representing the same object are related by a smooth unitary gauge transformation of determinant one. The charts and their transitions are independent of the Sobolev index $k\geq4$, and gauge transformations between smooth representatives are smooth. The results on graded representatives and open quotient charts in~\cite[Propositions~3.7 and~3.10 and Theorem~3.9]{CaoUniversalPolystable} show that the quotient topology agrees with the topology defined by the relative Kuranishi charts.

\section{Orbit-type decompositions}\label{sec:local}
We compare the stabilizers of Higgs bundles in the complex gauge group with the stabilizers of the quadratic models from Section~\ref{sec:preliminaries}. The gauge and model partitions agree after taking connected components. The subspaces fixed by the stabilizers provide coordinates along these strata.

\subsection{Stabilizer conjugacy and connected components}

\begin{lem}\label{lem:compact-refinement}
Let a compact Lie group $K$, possibly disconnected, act smoothly on a finite-dimensional manifold $M$. Let $Z\subseteq M$ be any invariant subset, with its subspace topology, and give $Z/K$ the quotient topology. Let $\operatorname{lab}$ assign labels in a set $\mathfrak L$ to the stabilizers $K_z$. Assume that $\operatorname{lab}$ is invariant under $K$-conjugacy and that equal labels imply equality of real dimensions and numbers of connected components.

For $\ell\in\mathfrak L$, consider the locus
\begin{equation*}
 \{[z]\in Z/K:\operatorname{lab}(K_z)=\ell\}.
\end{equation*}
Every $K$-orbit-type locus is relatively open and closed in the locus with the same label. On every open subset of $Z/K$, the partitions by $K$-orbit types and by labels have identical connected refinements.

\end{lem}

\begin{proof}
For a closed inclusion $L_1\subseteq L_2$ of compact Lie groups,
\begin{equation*}
 \dim L_1=\dim L_2,
 \quad |\pi_0(L_1)|=|\pi_0(L_2)|
 \quad\Longrightarrow\quad
 L_1^0=L_2^0,
 \quad L_1=L_2.
\end{equation*}
Indeed, the dimension equality gives $\Lie L_1=\Lie L_2$ and hence $L_1^0=L_2^0$; the components of both groups are cosets of this common identity component, and the inclusion together with the component count exhausts the cosets of $L_2$.

Fix $z\in Z$. Averaging a Riemannian metric over $K$ gives an invariant metric on $M$. An invariant tubular neighborhood of the compact orbit $Kz$ has an equivariant projection onto $Kz$, as in~\cite[proof of Lemma~4.1]{Mayrand2020Stratification}. Its intersection $N$ with $Z$ is invariant and relatively open. If the projection of $w\in N$ is $kz$, then $K_w\subseteq kK_zk^{-1}$. When $\operatorname{lab}(K_w)=\operatorname{lab}(K_z)$, the preceding inclusion argument gives $K_w=kK_zk^{-1}$.

The image of $N$ is open in $Z/K$, since its inverse image is $N$. Thus the stabilizer conjugacy class is locally constant on every locus with a given label. Each orbit-type locus is relatively open in that locus, and its relative complement is a union of other orbit-type loci. After restriction to an arbitrary open subset of the quotient, every connected subset on which the label is constant lies in a single orbit-type locus. Maximality of connected components proves the assertion.
\end{proof}

\begin{prop}\label{prop:conjugacy}
Let $K_i$ and $H_i=K_i\exp(i\Lie K_i)$ be the stabilizers of two smooth solutions to the Hitchin equations on the reference bundle. Then
\begin{equation*}
 \bigl(\exists g\in\cG_{k+1}^{\C}:\ gH_1g^{-1}=H_2\bigr)
 \quad\Longleftrightarrow\quad
 \bigl(\exists u\in\cG_{k+1}:\ uK_1u^{-1}=K_2\bigr).
\end{equation*}
The same equivalence holds for smooth gauge transformations and is independent of the chosen solutions to the Hitchin equations.
\end{prop}

\begin{proof}
Let $\mathcal A_i$ be the algebra of Higgs endomorphisms of the $i$th solution. We extend the argument for conjugacy of compact stabilizers in~\cite[Theorem~3.1 and Corollary~3.3]{FanOrbitTypes} to the gauge groups of determinant one.

Suppose $gH_1g^{-1}=H_2$ and set $C=gK_1g^{-1}$. Normalized Haar measure on $C$ defines
\begin{equation*}
 M=\int_C c^*c\,dc .
\end{equation*}
The endomorphism $M$ belongs to the parallel algebra $\mathcal A_2$, is positive, and satisfies
\begin{equation*}
 c_0^*Mc_0=M\qquad(c_0\in C).
\end{equation*}
Since $M$ is parallel, $\det M$ is constant on $\Sigma$. Functional calculus in $\mathcal A_2$ gives the positive endomorphism of determinant one $P=(\det M)^{-1/(2r)}M^{1/2}\in H_2$. The identity above gives $C'=PCP^{-1}\subseteq K_2$. Indeed, for $c\in C$,
\begin{equation*}
 (PcP^{-1})^*(PcP^{-1})
 =P^{-1}c^*P^2cP^{-1}=\Id.
\end{equation*}
Conjugation by $Pg$ gives $H_2=C'\exp(i\Lie C')$. If an element of this expression lies in $K_2$, uniqueness of its matrix polar decomposition forces its positive factor to be the identity. Hence $C'=K_2$, including all components.

We set $g_0=Pg$ and $h_0=(g_0^*g_0)^{1/2}$. Since $g_0K_1g_0^{-1}=K_2$, the positive endomorphism $g_0^*g_0$ and its square root $h_0$ commute with $K_1$. Hence $u=g_0h_0^{-1}$ is unitary, conjugates $K_1$ to $K_2$, and has determinant one because $\det h_0=|\det g_0|=1$. Conversely, a unitary transformation conjugating $K_1$ to $K_2$ also conjugates $H_1$ to $H_2$ by polar decomposition.

The algebra $W^{k+1,2}$ is closed under multiplication, inversion, and smooth functional calculus on uniformly positive matrices. Hence the conjugator constructed above remains of class $W^{k+1,2}$. The averaging defining $M$ takes place in a finite-dimensional algebra of smooth sections. A smooth $g$ gives a smooth $u$.

A Sobolev conjugator of the stabilizer representations can be replaced by a smooth conjugator as follows. For a unitary conjugator $u$, let $\gamma(t)=utu^{-1}$, $t\in K_1$. We approximate $u$ uniformly by a smooth endomorphism $s_0$ and set
\begin{equation*}
 s=\int_{K_1}\gamma(t)s_0t^{-1}\,dt.
\end{equation*}
The stabilizer representations consist of smooth sections in finite-dimensional spaces, so $s$ is smooth. Haar invariance gives $st=\gamma(t)s$, while
\begin{equation*}
 \|s-u\|_{C^0}\leq\|s_0-u\|_{C^0}.
\end{equation*}
We choose the approximation so that $s$ is invertible and $\det s$ lies in a disk about $1$ avoiding zero. The equivariant endomorphism
\begin{equation*}
 s^\circ=\exp\!\left(-\frac1r\log(\det s)\right)s,
 \qquad s^\circ t=\gamma(t)s^\circ,
\end{equation*}
is smooth and has determinant one. Its unitary polar factor is a smooth conjugator of determinant one. This proves independence of Sobolev order and agreement with smooth conjugacy. Changing the solution representing a fixed object conjugates its compact stabilizer by a smooth unitary gauge transformation of determinant one, by Section~\ref{sec:topology}. The conjugacy classes are therefore independent of that choice.
\end{proof}

To compare the local orbit-type partitions, we set $H=H_x$ and $K=K_x$. We choose a $K_x$-invariant Hermitian norm on $V_x$ and let $\mu_{x,\R}\colon V_x\to\fk_x^*$ be the real moment map for the resulting unitary representation, with $\mu_{x,\R}(0)=0$. We denote the set of minimal vectors in $Q_x$ by $Q_x^{\min}=Q_x\cap\mu_{x,\R}^{-1}(0)$. We apply the finite-dimensional quotient theorem~\cite[Theorem~2.12 and Section~4.2]{Mayrand2020Stratification} to $V_x$. By~\cite[Proposition~2.11(ii) and Section~4.1]{Mayrand2020Stratification}, the restriction of $q_x$ to $Q_x^{\min}$ induces a homeomorphism $Q_x^{\min}/K\xrightarrow{\sim}Y_x$. Every closed $H$-orbit meets $Q_x^{\min}$ in one $K$-orbit. Here $V_x$ is a unitary representation, $H=K^\C$ is reductive, and $Q_x$ is a closed invariant reduced affine subvariety. The theorem allows $H$ to be disconnected.

\begin{prop}\label{prop:types}
Let $\mathcal P_H$ be the orbit-type partition of $Y_x$ in Section~\ref{sec:deformation}, and let $\mathcal P_{\mathrm{gauge}}$ be the partition of $\cM$ into the loci defined by conjugacy in the complex gauge group, as specified in Definition~\ref{defn:types}. In the charts of Section~\ref{sec:relative-families},
\begin{equation*}
 \Comp\!\left(
   \chi_x^{-1}(\mathcal P_{\mathrm{gauge}}|_{\mathscr U_x})
 \right)
 =
 \Comp\!\left(
   \{B_x\times\mathscr R:
      \mathscr R\in\mathcal P_H|_{\mathcal W_x}\}
 \right).
\end{equation*}
The corresponding statement holds on every fiber over $b\in B_x$.
\end{prop}

\begin{proof}
For $w\in Q_x^{\min}$, every $h\in H_w$ has a polar decomposition $h=k\exp(i\xi)$. The function $f(t)=\|\exp(it\xi)w\|^2$ satisfies $f'(0)=0$, $f(1)=f(0)$, and
\begin{equation*}
 f''(t)=4\|i\xi\cdot(\exp(it\xi)\cdot w)\|^2\geq0.
\end{equation*}
Convexity forces $\xi w=0$, and then $kw=w$. Consequently
\begin{equation*}
\begin{gathered}
 H_w=K_w\exp(i\Lie K_w),\\
 \dim_\C H_w=\dim_\R K_w,\qquad
 |\pi_0(H_w)|=|\pi_0(K_w)|.
\end{gathered}
\end{equation*}
The last equalities follow from uniqueness of polar factors.

We label $K_w$ by the $\cG_{k+1}^{\C}$-conjugacy class of $H_w$:
\begin{equation*}
 \operatorname{lab}(K_w)
   =[H_w]_{\cG_{k+1}^{\C}}.
\end{equation*}
This label is invariant under $K$-conjugacy. Equal labels give conjugate complex Lie subgroups, hence equal dimensions and numbers of connected components. The polar decomposition therefore allows us to apply Lemma~\ref{lem:compact-refinement}. Each $K$-orbit-type locus is relatively open and closed in the locus with the corresponding label. Every $H$-orbit-type locus is a union of these $K$-orbit-type loci and is contained in a locus with a single label, where it is also relatively open and closed. Their connected refinements agree after restriction to any open subset.

For a representative $v$ of the closed orbit over $y\in\mathcal W_x$, the stabilizer of $\widetilde\theta_x(b,v)$ in the complex gauge group is $H_v$ for every $b$, by~\eqref{eq:stabilizer-comparison}. The representative $h\cdot v$ has stabilizer $hH_vh^{-1}$. Consequently, the pullback of each orbit-type locus for the complex gauge group has the form $B_x\times L$, where $L\subset\mathcal W_x$ is the locus with the corresponding label. Since $B_x$ is connected,
\begin{equation*}
 \Comp(B_x\times L)
   =\{B_x\times C:C\in\Comp(L)\}.
\end{equation*}
Projection sends every connected subset of $B_x\times L$ into one component of $L$, and each product $B_x\times C$ is connected. This proves the asserted equality of partitions.

Changing a polystable representative conjugates its stabilizer; for solutions to the Hitchin equations, the compact stabilizers are conjugate by a smooth unitary gauge transformation, by Section~\ref{sec:topology}. A smooth complex gauge transformation of determinant one takes each model representative to such a solution. Proposition~\ref{prop:conjugacy} provides a smooth conjugator there, and conjugating back relates the original representatives. Smooth and $W^{k+1,2}$ gauge conjugacy therefore define the same partition of the local model. On chart overlaps, both connected refinements are obtained by restricting the same orbit-type loci and taking connected components. The comparison is consequently independent of the chart and of all representative choices.
\end{proof}

\subsection{Fixed subspaces in the local quotient}\label{sec:fixed-factors}
The polar decomposition gives $V_x^{K_x}=V_x^{H_x}$. This fixed subspace is symplectic. Indeed, averaging over $K_x$ gives a projection $\mathsf{Av}_x\colon V_x\to V_x^{H_x}$, and for $u\in V_x^{H_x}$ invariance gives $\Omega_{x,\C}(u,z)=\Omega_{x,\C}(u,\mathsf{Av}_x z)$. Nondegeneracy of $\Omega_{x,\C}$ therefore implies nondegeneracy on $V_x^{H_x}$.

Let $N_x=(V_x^{H_x})^{\Omega_{x,\C}}$. The splitting of the symplectic $H_x$-representation $V_x$ into $V_x^{H_x}$ and $N_x$ gives the quotient decomposition~\cite[Section~4.4 and Lemma~4.3]{Mayrand2020Stratification}
\begin{equation}\label{eq:fixed-product}
\begin{aligned}
 Y_x&=R_x\times V_x^{H_x},\\
 R_x&=\left(\Spec\C[(\nu_{x,\C}|_{N_x})^{-1}(0)_{\red}]^{H_x}\right)^{\an},\\
 q_x(f)&=(0_{R_x},f),
\end{aligned}
\end{equation}
where $0_{R_x}$ is the image of zero in the factor $R_x$. Thus $q_x|_{V_x^{H_x}}$ is a closed holomorphic embedding. The fixed vectors lie in $Q_x$ and have singleton closed orbits.

Together with the curve coordinate, the fixed subspace $V_x^{H_x}$ gives local coordinates on the orbit-type stratum through $x$ in a sufficiently small quotient chart.

\begin{prop}\label{prop:fixed-factor}
Every connected orbit-type stratum $\mathscr S$ is a locally closed complex submanifold with its reduced induced structure, and $p|_{\mathscr S}$ is a holomorphic submersion. At $x\in\mathscr S$, after shrinking the quotient chart, there is a connected open ball $O_x\subset U_x\cap V_x^{H_x}$ about zero such that
\begin{equation*}
 \mathscr S\cap\mathscr U_x
   =\chi_x(B_x\times q_x(O_x)).
\end{equation*}
This is a biholomorphic parametrization by $B_x\times O_x$, over $B_x$. In these coordinates, the vertical tangent bundle is trivial with fiber $V_x^{H_x}$.
\end{prop}

\begin{proof}
Let $\mathscr L_x$ denote the orbit-type locus for the complex gauge group that contains $x$. For a representative $v$ of a closed orbit in the chart, we set $J=(H_x)_v$. This is a closed algebraic subgroup of $H_x$. If the class of $\widetilde\theta_x(b,v)$ lies in $\mathscr L_x$,~\eqref{eq:stabilizer-comparison} gives a conjugacy in the complex gauge group between $J$ and $H_x$. Thus $\dim_\C J=\dim_\C H_x$ and they have the same number of connected components. Inclusion gives $\Lie J=\Lie H_x$, hence $J^0=H_x^0$. Both component sets consist of cosets of that subgroup, so their equal cardinalities imply $J=H_x$. It follows that $v\in V_x^{H_x}$. The converse follows directly from~\eqref{eq:stabilizer-comparison}. Therefore the locus in the chart with stabilizer $H_x$ is
\begin{equation*}
 B_x\times\bigl(\mathcal W_x\cap q_x(V_x^{H_x})\bigr).
\end{equation*}

We choose an open neighborhood $R_x'\subset R_x$ of $0_{R_x}$ and set $\mathcal W_x=R_x'\times O_x$ in~\eqref{eq:fixed-product}, with $O_x$ and $B_x$ connected. The locus with stabilizer $H_x$ is then $B_x\times\{0_{R_x}\}\times O_x$. It is connected, contains $x$, and is closed and nonsingular in the chart, so it is contained in $\mathscr S$. Conversely, $\mathscr S\cap\mathscr U_x$ lies in the locus with stabilizer $H_x$. Hence $\mathscr S\cap\mathscr U_x=\chi_x(B_x\times q_x(O_x))$.

Applying this construction in a chart centered at each point gives the reduced complex manifold structure. These structures agree because they are induced from the reduced analytic structure of $\cM$. The stratum is closed in the union of these charts centered on $\mathscr S$, so it is locally closed. The projection is $B_x\times O_x\to B_x$, and the kernel of its differential is the tangent bundle along the $O_x$-factor. Proposition~\ref{prop:types} gives independence of all choices.
\end{proof}

\section{Whitney stratification}\label{sec:whitney}
Proposition~\ref{prop:fixed-factor} has established the complex manifold structure of the orbit-type strata. We study the connected components of their intersections with quotient charts and prove Theorem~\ref{thm:main-strata}.

The finite-dimensional quotient in Section~\ref{sec:deformation} has its complex Whitney stratification by~\cite[Theorem~1.4(i), Section~4.2, and Proposition~4.10]{Mayrand2020Stratification}. The algebraic symplectic $H_x$-action is the holomorphic extension of the $K_x$-action. The same result applies to $N_x$ in~\eqref{eq:fixed-product}. Its subspace fixed by $H_x$ is zero, so the stratum through $0_{R_x}$ is $\{0_{R_x}\}$. For $n\in N_x\cap Q_x$ and $f\in V_x^{H_x}$, the stabilizer of $n+f$ equals that of $n$, and $H_x\cdot(n+f)$ is closed exactly when $H_x\cdot n$ is closed. Thus the connected model strata are $\mathscr R_\alpha\times V_x^{H_x}$, where $\mathscr R_\alpha$ is a connected orbit-type stratum in $R_x$.

\begin{lem}\label{lem:local-neighborhoods}
At every $x$, the quotient chart can be chosen with $\mathcal W_x=R_\varepsilon\times O_x$, where $R_\varepsilon$ is a neighborhood of $0_{R_x}$ in $R_x$ and $B_x,O_x$ are connected. The resulting product chart has finitely many connected local model strata
\begin{equation*}
 B_x\times(\mathscr R_\alpha\cap R_\varepsilon)\times O_x.
\end{equation*}
Each $\mathscr R_\alpha\cap R_\varepsilon$ is nonempty and path connected. The global stratum through $x$ meets the chart in $B_x\times\{0_{R_x}\}\times O_x$, and the closure of every local model stratum contains this intersection. Such charts can be chosen arbitrarily small.
\end{lem}

\begin{proof}
We choose positive-degree homogeneous generators $z_1,\ldots,z_{N_R}$ of the invariant algebra defining $R_x$, with degrees $d_1,\ldots,d_{N_R}$. The generators define a closed embedding $R_x\hookrightarrow\C^{N_R}$ with $0_{R_x}=0$. Scalar multiplication on $N_x$ descends to
\begin{equation*}
 z_j(d_ty)=t^{d_j}z_j(y),\qquad t>0,
\end{equation*}
and extends continuously to $d_0y=0_{R_x}$. For positive $t$, the map $d_t$ preserves closed orbits and stabilizers. The path $t\mapsto d_ty$ remains in the same orbit-type stratum. Every orbit-type stratum $\mathscr R_\alpha\subset R_x$ is invariant under these dilations and has $0_{R_x}$ in its closure. Local finiteness at $0_{R_x}$ therefore implies that only finitely many orbit-type strata occur in $R_x$.

For positive numbers $\varepsilon_j$, we define
\begin{equation*}
 R_\varepsilon=\{y\in R_x:|z_j(y)|<\varepsilon_j\text{ for all }j\}.
\end{equation*}
These neighborhoods can be arbitrarily small and are preserved by $d_t$ for $0<t\leq1$. Every compact subset is carried into $R_\varepsilon$ by sufficiently small positive $t$. If $N_R=0$, we set $R_\varepsilon=R_x=\{0_{R_x}\}$. For each $\mathscr R_\alpha$, the intersection $\mathscr R_\alpha\cap R_\varepsilon$ is nonempty and path connected. Indeed, any two points of the intersection can be joined by a path in the connected complex manifold $\mathscr R_\alpha$. A sufficiently small dilation carries that compact path into $R_\varepsilon$, and radial paths join its dilated endpoints to the original endpoints. All three paths lie in $\mathscr R_\alpha\cap R_\varepsilon$.

We choose a chart with $\mathcal W_x=R_\varepsilon\times O_x$ and connected $B_x,O_x$. Its connected local model strata are
\begin{equation}\label{eq:local-strata}
 B_x\times(\mathscr R_\alpha\cap R_\varepsilon)\times O_x .
\end{equation}
The global stratum $\mathscr S_x$ through $x$ meets this chart in $B_x\times\{0_{R_x}\}\times O_x$, by Proposition~\ref{prop:fixed-factor}. The closure of every stratum in~\eqref{eq:local-strata} contains this intersection. For fixed points in $B_x$ and $O_x$, the dilation $d_t$ sends the $R_x$-coordinate to $0_{R_x}$ as $t\to0$.
\end{proof}

\begin{proof}[Proof of Theorem~\ref{thm:main-strata}]
We use the charts from Lemma~\ref{lem:local-neighborhoods}. By Proposition~\ref{prop:types}, the sets in~\eqref{eq:local-strata} are exactly the connected components of the orbit-type loci restricted to the chart. For any partition $\mathcal P$ and open $U$,
\begin{equation*}
 \Comp(\Comp(\mathcal P)|_U)=\Comp(\mathcal P|_U).
\end{equation*}
A connected component of the intersection of a locus with $U$ lies in one global connected component of the locus; maximality of connected components gives the equality in both directions. Each global stratum meeting the chart contains at least one of the finitely many local strata in~\eqref{eq:local-strata}. This proves local finiteness of the global partition.

The frontier criterion~\cite[Lemma~4.6]{Mayrand2020Stratification} applies directly to this partition. At every $x$, the stratum through $x$ has connected intersection with the chosen neighborhood, and every connected component of a restricted stratum has that intersection in its closure. Since the global strata are connected by definition, the criterion gives the stated global frontier property.

The finite partition~\eqref{eq:local-strata} satisfies Whitney \textup{(a)} and \textup{(b)}. Along a stratum $\mathscr R_\alpha\subset R_x$, at $(b,y,f)\in B_x\times\mathscr R_\alpha\times O_x$,
\begin{equation*}
 T_{(b,y,f)}
   \bigl(B_x\times\mathscr R_\alpha\times O_x\bigr)
 =T_bB_x\oplus T_y\mathscr R_\alpha\oplus T_fO_x.
\end{equation*}
For a sequence $(b_n,y_n,f_n)\to(b,y,f)$, if $T_{y_n}\mathscr R_\alpha\to T_\infty^R$, then the corresponding product tangent spaces converge to $T_bB_x\oplus T_\infty^R\oplus T_fO_x$. The $R_x$-component of a limiting secant is zero or a scalar multiple of an $R_x$ secant limit, and Whitney \textup{(b)} places that component in $T_\infty^R$. The other components of the limiting secant lie in $T_bB_x\oplus T_fO_x$. The same direct sum decomposition gives Whitney \textup{(a)}. The conditions are invariant under changes of analytic embeddings~\cite[Definition~2.7 and Section~4.7]{Mayrand2020Stratification}.

Now let $x\in\mathscr S\cap\overline{\mathscr R}$ for global strata $\mathscr S$ and $\mathscr R$. The stratum through $x$ in a chart centered at $x$ is the product $B_x\times\{0_{R_x}\}\times O_x$. Let $x_n\in\mathscr R$ and $y_n\in\mathscr S$ tend to $x$, and suppose that the tangent spaces $T_{x_n}\mathscr R$ and the secant lines through $x_n,y_n$ converge in a local analytic embedding. Finiteness of the local partition provides a subsequence for which all $x_n$ lie in one stratum of~\eqref{eq:local-strata}, with unchanged limits. That local stratum is a connected component of the open subset $\mathscr R\cap\mathscr U_x$ of the manifold $\mathscr R$, so their tangent spaces agree. The product Whitney conditions prove the conditions for the original limits. Proposition~\ref{prop:fixed-factor} supplies the complex manifold structure and the submersion property.
\end{proof}

The same charts control the vertical tangent spaces near a lower-dimensional stratum. We use Thom's condition~$(a_p)$ as in~\cite[Section~11]{Mather2012}. For strata $\mathscr R,\mathscr S$, a point $x\in\mathscr S\cap\overline{\mathscr R}$, and a sequence $x_n\in\mathscr R$ tending to $x$, the condition is
\begin{equation}\label{eq:thom-condition}
 \ker d(p|_{\mathscr R})_{x_n}\longrightarrow T_\infty
 \quad\Longrightarrow\quad
 \ker d(p|_{\mathscr S})_x\subseteq T_\infty,
\end{equation}
where the limit is taken in the appropriate Grassmannian in a local analytic embedding.

\begin{cor}\label{cor:thom}
The stratification $\mathcal P_{\mathrm{conn}}$ satisfies Thom's condition~$(a_p)$ for $p\colon\cM\to\cT$.
\end{cor}

\begin{proof}
We use the chart of Lemma~\ref{lem:local-neighborhoods} centered at $x$. In product coordinates, the stratum through $x$ is $B_x\times\{0_{R_x}\}\times O_x$. After passing to a subsequence, the points $x_n$ lie in one of the finitely many local strata $B_x\times(\mathscr R_\alpha\cap R_\varepsilon)\times O_x$. At a point $(b,y,f)$ of this stratum, the vertical tangent space is
\begin{equation*}
 \{0\}\oplus T_y\mathscr R_\alpha\oplus V_x^{H_x}.
\end{equation*}
Every such space contains $\{0\}\oplus\{0\}\oplus V_x^{H_x}$, which is the vertical tangent space of the stratum through $x$. This inclusion persists under Grassmannian limits and proves~\eqref{eq:thom-condition}.
\end{proof}

\section{Harmonic lifts}
\label{sec:harmonic}
We construct harmonic lifts on the charts parametrized by $B_x\times O_x$, with $O_x\subset V_x^{H_x}$, from Proposition~\ref{prop:fixed-factor}. The endomorphism algebra remains constant along these families and gives the rank formulas required for parameter-dependent Hodge theory. Unitary gauge transformations preserving the determinant trivialization compare the harmonic families on chart overlaps and complete the descent in Theorem~\ref{thm:main-harmonic}.

\subsection{Harmonic lifts and stable factors}
We fix a stratum $\mathscr S$, a point $x\in\mathscr S$, and the $H_x$-fixed ball $O_x\subset V_x^{H_x}$ from Proposition~\ref{prop:fixed-factor}. We set $b_x=p(x)$ and $u_x=(b_x,0)\in B_x\times O_x$. For the Kuranishi family $\theta_x(b,v)$ on $B_x\times O_x$, let $\mathbf c_x(b,v)=(\nabla_{h,b,v},\varphi_{b,v})$ be the pair consisting of the Chern connection determined by the fixed metric $h$ and the Higgs field of $\theta_x(b,v)$. The actions are related by~\eqref{eq:relative-equivariance}.

\begin{prop}\label{prop:harmonic-lifts}
After shrinking $B_x\times O_x$, there is a positive complex gauge transformation of determinant one $g_x^+(u)$, where $u=(b,v)$, that commutes with $H_x$. The pair
\begin{equation*}
 \widehat c_x(u)=\mathbf c_x(u)^{g_x^+(u)}
\end{equation*}
solves the Hitchin equations on $X_b=(\Sigma,j_b)$. Its complex and compact stabilizers are exactly $H_x$ and $K_x$. In the fixed smooth trivialization, the coefficients of $\theta_x,g_x^+,\widehat c_x$ and the harmonic metric
\begin{equation*}
 h_u(s,t)=h((g_x^+(u))^{-1}s,(g_x^+(u))^{-1}t)
\end{equation*}
are smooth on $B_x\times O_x\times\Sigma$. At $u_x=(b_x,0)$, the vertical differential of $\widehat c_x$ is the inclusion $V_x^{H_x}\hookrightarrow V_x$, after identifying a connection variation with its $(0,1)$ part.
\end{prop}

\begin{proof}
We choose metrics compatible with $j_b$ and a fixed area form, and denote their Hodge star by $\star_b$. These metrics vary smoothly with $b$. On zero-forms and two-forms, the stars agree because the area form is fixed. The real moment-map correction in~\cite[Proposition~3.1]{CaoUniversalPolystable} gives a smooth map $\beta_x(b,v)$ and the positive gauge transformation of determinant one
\begin{equation*}
 g_x^+(b,v)=
 \exp\bigl(-i\star_b\beta_x(b,v)\bigr).
\end{equation*}
The map $g_x^+$ is $K_x$-equivariant, and its coefficients are smooth on $B_x\times O_x\times\Sigma$. For
\begin{equation*}
 \boldsymbol\mu_{\R}(\nabla,\varphi)
       =F_\nabla+[\varphi,\varphi^{*h}],
\end{equation*}
the real moment map evaluated on the corrected family is
\begin{equation*}
 \nu_{x,\R}(b,v)
   :=\boldsymbol\mu_{\R}(\widehat c_x(b,v)).
\end{equation*}
The correction equation places $\nu_{x,\R}(b,v)$ in the fixed space $\star_b\fk_x$. Let $\omega_b$ be the pullback of the $L^2$ K\"ahler form on the configuration space along $v\mapsto\widehat c_x(b,v)$. It is a closed real two-form. As in~\cite[Section~3.2]{CaoUniversalPolystable}, set
\begin{equation*}
 \langle\ell_b(v),\xi\rangle
   =-\int_\Sigma\tr(\xi\nu_{x,\R}(b,v)),\qquad \xi\in\fk_x.
\end{equation*}
The right gauge action, the left model action, and~\cite[Lemma~3.2]{CaoUniversalPolystable} give
\begin{equation}\label{eq:fixed-moment-identity}
 d_v\langle\ell_b,\xi\rangle
   =\iota_{[\xi,v]}\omega_b,
 \qquad \ell_b(0)=0.
\end{equation}
The trace pairing between $\star_b\fk_x$ and $\fk_x$ is nondegenerate. On the reduced equation locus $Z_x$, the corrected pair solves the Hitchin equations exactly when $\nu_{x,\R}$ vanishes. The correction construction also gives $d_v\beta_x|_{u_x}=0$.

For $\xi\in\fk_x$ and $v\in O_x$, the vector $tv$ is fixed by $\xi$. Equation~\eqref{eq:fixed-moment-identity} therefore gives
\begin{equation*}
 \langle\ell_b(v)-\ell_b(0),\xi\rangle
 =\int_0^1\omega_b|_{tv}([\xi,tv],v)\,dt=0.
\end{equation*}
Thus $\ell_b(v)=0$, and nondegeneracy of the trace pairing gives $\nu_{x,\R}(b,v)=0$. Moreover $V_x^{H_x}\subset Q_x$ by~\eqref{eq:moment}. Thus $\widehat c_x$ is a solution for every $(b,v)\in B_x\times O_x$.

At every $K_x$-fixed parameter, equivariance implies that $g_x^+$ commutes with $K_x$, and hence with $H_x=K_x\exp(i\fk_x)$. For the right action, the stabilizer transforms by inverse conjugation. The identity~\eqref{eq:stabilizer-comparison}, restricted to the common domain of $\theta_x$ and $\widetilde\theta_x$, therefore gives
\begin{equation*}
 \Stab_{\cG_{k+1}^{\C}}(\widehat c_x)
      =(g_x^+)^{-1}H_xg_x^+=H_x.
\end{equation*}
Intersection with the unitary group gives $K_x$. Transporting $h$ gives the displayed harmonic metric. The asserted smoothness follows from the correction construction, the regularity of $A_x$, and smooth inversion. The vertical Kuranishi differential $d_vA_x|_{u_x}$ is the harmonic inclusion, and the correction has zero vertical differential at $u_x$. This proves the tangent assertion.
\end{proof}

\begin{lem}\label{lem:constant-endomorphism-algebra}
Let
\begin{equation*}
 \mathcal A_u=
 \{e\in H^0(X_b,\End E_u):[\varphi_u,e]=0\}
\end{equation*}
be the algebra of Higgs endomorphisms of $\theta_x(u)$. Under the fixed smooth trivialization, one has
\begin{equation*}
 \mathcal A_u=\mathcal A_x
     =\C\Id+\Lie H_x=\operatorname{span}_{\C}H_x
\end{equation*}
for every $u\in B_x\times O_x$.
\end{lem}

\begin{proof}
Every invertible element $s\in\mathcal A_x$ is parallel, so $\det s$ is a nonzero constant on the connected curve. A scalar multiple of $s$ therefore has determinant one and belongs to $H_x$. For any $e\in\mathcal A_x$, one may choose $\lambda\in\C^*$ outside the finite spectrum of $e$; then
\begin{equation*}
 e=(e-\lambda\Id)+\lambda\Id
\end{equation*}
is a sum of two invertible elements. Hence $\operatorname{span}_{\C}H_x=\mathcal A_x$.

Now let $e\in\mathcal A_u$. Its trace is a holomorphic function on the compact connected curve $X_b$ and is therefore constant. The trace-free endomorphism
\begin{equation*}
 e_0=e-\frac{\tr e}{r}\Id
\end{equation*}
exponentiates to a Higgs automorphism preserving the determinant trivialization. Equation~\eqref{eq:stabilizer-comparison} therefore gives $\exp(te_0)\in H_x$ for every $t\in\C$. Differentiation at $t=0$ gives $e_0\in\Lie H_x$ and $\mathcal A_u\subseteq\C\Id+\Lie H_x$. The same argument at the base point gives $\mathcal A_x=\C\Id+\Lie H_x$. Conversely, $H_x$ acts by Higgs automorphisms on every member of the family over $B_x\times O_x$, so $\mathcal A_x=\operatorname{span}_{\C}H_x\subseteq\mathcal A_u$. These two inclusions prove the lemma.
\end{proof}

For each matrix block of $\mathcal A_x$, we fix standard matrix units $e_{\alpha,ij}$ satisfying
\begin{equation*}
 e_{\alpha,ij}e_{\beta,pq}
   =\delta_{\alpha\beta}\delta_{jp}e_{\alpha,iq},
 \qquad \sum_{\alpha,i}e_{\alpha,ii}=\Id.
\end{equation*}

\begin{prop}\label{prop:stable-decomposition}
The images of the matrix units define holomorphic Higgs subfamilies $E_{\alpha,u}$ and decompositions $E_u=\bigoplus_\alpha E_{\alpha,u}\otimes V_\alpha$, with
\begin{equation*}
 \varphi_u=\bigoplus_\alpha
       \varphi_{\alpha,u}\otimes\Id_{V_\alpha}.
\end{equation*}
The factors are stable of degree zero and pairwise nonisomorphic. The harmonic metric has the block form $\bigoplus_\alpha h_{\alpha,u}\otimes h_{V_\alpha}$, where the multiplicity metrics $h_{V_\alpha}$ are fixed. The determinant and trace conditions are
\begin{equation*}
 \bigotimes_\alpha(\det E_{\alpha,u})^{m_\alpha}
          \otimes(\det V_\alpha)^{d_\alpha}\simeq\cO
       \quad\text{through }\tau,\qquad
 \sum_\alpha m_\alpha\tr(\varphi_{\alpha,u})=0.
\end{equation*}
The determinant line bundles and traces of individual factors may vary.
\end{prop}

\begin{proof}
By equivariance of $\theta_x$, every element of $H_x$ acts as a holomorphic automorphism of the family over $B_x\times O_x$. Lemma~\ref{lem:constant-endomorphism-algebra} gives $\mathcal A_x=\operatorname{span}_{\C}H_x$, so each $e_{\alpha,ij}$ is a holomorphic endomorphism of the total family. In particular,
\begin{equation*}
 \bar\partial_{\Hom}e_{\alpha,ij}=0,
 \qquad [\varphi_u,e_{\alpha,ij}]=0.
\end{equation*}
Each idempotent $e_{\alpha,11}$ has constant rank; its image is therefore a holomorphic subbundle. The maps $e_{\alpha,i1}$ are isomorphisms between the equivalent summands and give $E_u=\bigoplus_\alpha E_{\alpha,u}\otimes V_\alpha$.

At every parameter the total object is polystable of slope zero. A morphism between stable Higgs bundles of equal slope is zero or an isomorphism, by stability applied to its kernel and saturated image. Every endomorphism of a stable Higgs bundle is scalar. Its characteristic coefficients are constant; subtracting an eigenvalue times the identity gives a noninvertible endomorphism, which is zero by the preceding argument. Thus an idempotent image of a polystable object is polystable of slope zero, and its endomorphism algebra is a product of multiplicity matrix algebras. Since $e_{\alpha,11}\mathcal A_ue_{\alpha,11}=\C e_{\alpha,11}$, the factor $E_{\alpha,u}$ is stable. For $\alpha\ne\beta$, the equality $e_{\alpha,11}\mathcal A_ue_{\beta,11}=0$ shows that there are no nonzero Higgs morphisms from $E_{\beta,u}$ to $E_{\alpha,u}$. The factors are therefore pairwise nonisomorphic.

The fiberwise commutant of $\mathcal A_x$ at $u$ is $\bigoplus_\alpha\End(E_{\alpha,u})\otimes\Id_{V_\alpha}$. Since $g_x^+$ commutes with $\mathcal A_x$, it is block diagonal with respect to the stable isotypic decomposition. The transported harmonic metric consequently has the block form stated in the proposition. The harmonic equation restricts to each block. The determinant trivialization and the trace-free condition give the two displayed identities.
\end{proof}

\begin{cor}\label{cor:normalizer-blocks}
Every isomorphism of Higgs bundles preserving the determinant trivialization between two representatives whose stabilizer equals $H_x$ normalizes $\mathcal A_x$. It permutes the stable isotypic blocks and preserves the pairs $(d_\alpha,m_\alpha)$.
\end{cor}

\begin{proof}
Such an isomorphism conjugates the stabilizer $H_x$ to itself and therefore normalizes its complex span $\mathcal A_x$. It permutes the primitive central idempotents of $\mathcal A_x$. Isomorphic matrix blocks have the same size $m_\alpha$, and equality of the corresponding bundle ranks $d_\alpha m_\alpha$ gives equality of $d_\alpha$.
\end{proof}

\subsection{Hodge theory and comparison on overlaps}
We write $u=(b,v)$ and work with the lift $\widehat c_x(u)$ from Proposition~\ref{prop:harmonic-lifts}. Chosen smooth identifications of the $(p,q)$-form bundles for $j_b$ place the Higgs differentials $D_{0,u},D_{1,u}$ and the Laplacians $\Delta_{j,u}$ on fixed bundles over $\Sigma$. We write $D_u$ for the total differential determined by $D_{0,u}$ and $D_{1,u}$. The coefficients of these operators depend smoothly on $u$. Every element of $K_x$ is parallel for the connection in each member of the family. Conjugation by $K_x$ commutes with the differentials, adjoints and Laplacians. The coefficient bundle $\mathscr C_x=\End_0(E)^{K_x}$ is the trace-free part of the commutant of $\mathcal A_x$. Proposition~\ref{prop:stable-decomposition} shows that it is a fixed smooth subbundle. Let $\Delta^{\mathscr C}_{j,u}$ denote the restriction of the degree-$j$ Laplacian to forms with coefficients in $\mathscr C_x$, and set
\begin{equation*}
 \mathcal H^j_u=\ker\Delta_{j,u},\qquad
 \mathcal H^j_{u,\mathrm{inv}}
   =\ker\Delta^{\mathscr C}_{j,u}=\bigl(\mathcal H^j_u\bigr)^{K_x}.
\end{equation*}

\begin{lem}\label{lem:hodge-ranks}
After shrinking $B_x$ and $O_x$ to connected neighborhoods, the harmonic spaces of the $\End_0E$-valued deformation complex and its $K_x$-invariant subcomplex have constant dimensions. If $s$ is the number of stable isotypic blocks, then
\begin{equation*}
 \dim_\C\mathcal H^0_{u,\mathrm{inv}}
 =\dim_\C\mathcal H^2_{u,\mathrm{inv}}=s-1,
 \qquad
 \dim_\C\mathcal H^1_{u,\mathrm{inv}}=\dim_\C V_x^{H_x}.
\end{equation*}
\end{lem}

\begin{proof}
The restricted Higgs complex is elliptic because its symbol is the Dolbeault symbol. Averaging over $K_x$ commutes with the differentials and their adjoints, so the harmonic spaces of the restricted complex are $(\mathcal H^j_u)^{K_x}$.

In degree zero, the harmonic space is $\mathcal H^0_u=\fh_x$. Its invariant part is the trace-free center of $\mathcal A_x$, of dimension $s-1$. The equivariant trace duality in Section~\ref{sec:deformation}, applied at each solution to the Hitchin equations, gives $(\mathcal H^2_u)^{K_x}\cong((\mathcal H^0_u)^{K_x})^*$~\cite[Lemma~2.15]{CaoUniversalPolystable}. Thus the invariant harmonic space in degree two also has dimension $s-1$. After the form bundles have been identified and the background norms have been fixed, the even-to-odd operator
\begin{equation*}
 \mathcal D_u=D_u+D_u^*
\end{equation*}
on $\mathscr C_x$ is a family of bounded Fredholm operators from the fixed $W^{1,2}$ space of even-degree forms to the fixed $L^2$ space of odd-degree forms. The smooth coefficients give continuity in operator norm on these fixed spaces. Consequently
\begin{equation*}
 \dim\mathcal H^1_{u,\mathrm{inv}}
 =\dim\mathcal H^0_{u,\mathrm{inv}}
  +\dim\mathcal H^2_{u,\mathrm{inv}}-\operatorname{ind}\mathcal D_u.
\end{equation*}
The index is locally constant~\cite[Chapter~IV, Theorem~5.22, p.~236]{KatoPerturbation}. At the base point the invariant degree-one space is $V_x^{H_x}$, so the displayed formula gives $\dim_\C\mathcal H^1_{u,\mathrm{inv}}=\dim_\C V_x^{H_x}$ throughout $B_x\times O_x$. For the unrestricted deformation complex, $\mathcal H^0_u=\fh_x$ and trace duality gives $\dim\mathcal H^2_u=\dim\fh_x$. Its Fredholm index is also locally constant, so its degree-one harmonic dimension is constant as well.
\end{proof}

\begin{prop}\label{prop:hodge}
After shrinking $B_x$ and $O_x$, for every integer $m\geq0$, the harmonic projections and Green operators for the deformation complex and its invariant subcomplex
\begin{equation*}
 \mathsf H_{j,u}\colon W^{m,2}\longrightarrow W^{m,2},\qquad
 G_{j,u}\colon W^{m,2}\longrightarrow W^{m+2,2}
\end{equation*}
depend smoothly on $u\in B_x\times O_x$ in the corresponding operator norms. The invariant harmonic spaces $\mathcal H^1_{u,\mathrm{inv}}$ form a smooth complex vector bundle $\mathcal H^1_{\mathrm{inv}}$ of rank $\dim_\C V_x^{H_x}$.
\end{prop}

\begin{proof}
We first work with the $L^2$ realizations and fixed $W^{2,2}$ domains. Smooth bundle isometries and multiplication by the square root of the density ratio identify the varying $L^2$ spaces. We choose a positively oriented circle $\Gamma$ about zero such that, at $u_x$, it meets none of the spectra and encloses no nonzero eigenvalue of the Laplacians on the deformation complex or its invariant subcomplex. When $\mathscr C_x=0$, the restricted projections and Green operators are zero. Norm continuity from $W^{2,2}$ to $L^2$ and a Neumann series argument give resolvents on $\Gamma$ after shrinking $B_x$ and $O_x$. The Riesz projection is
\begin{equation}\label{eq:riesz-harmonic-projection}
 \mathsf H_{j,u}
   =\frac{1}{2\pi i}\int_\Gamma(z-\Delta_{j,u})^{-1}\,dz.
\end{equation}
It has the same rank as the harmonic projection at $u_x$. By Lemma~\ref{lem:hodge-ranks}, the kernel already has this rank, so $\Gamma$ encloses no positive eigenvalue. The neighborhoods $B_x$ and $O_x$ chosen above are independent of $m$. On these neighborhoods, elliptic regularity gives the resolvents as bounded operators $W^{m,2}\to W^{m+2,2}$ for every integer $m\geq0$.

For each $m\geq0$, elliptic regularity implies that $z-\Delta_{j,u}$ is an isomorphism $W^{m+2,2}\to W^{m,2}$ for $z\in\Gamma$. The coefficients vary smoothly, and inversion on the open set of bounded invertible operators is smooth. The contour formula therefore gives smooth dependence of $\mathsf H_{j,u}$ at every Sobolev order. The operator $\Delta_{j,u}+\mathsf H_{j,u}\colon W^{m+2,2}\to W^{m,2}$ is invertible: it is the identity on the kernel and agrees with $\Delta_{j,u}$ on its orthogonal complement. Hence
\begin{equation}\label{eq:green-inverse}
 G_{j,u}=(\Delta_{j,u}+\mathsf H_{j,u})^{-1}(1-\mathsf H_{j,u})
\end{equation}
depends smoothly on $u$ as an operator $W^{m,2}\to W^{m+2,2}$. The inverses at different Sobolev orders agree by uniqueness, on the parameter neighborhoods already chosen. The same construction commutes with $K_x$ and therefore applies to the invariant complex. Formula~\eqref{eq:riesz-harmonic-projection} and the constant rank in Lemma~\ref{lem:hodge-ranks} show that the invariant degree-one ranges form a smooth vector bundle of rank $\dim_\C V_x^{H_x}$.
\end{proof}

\begin{lem}\label{lem:unitary-overlaps}
On chart overlaps, the lifts of Proposition~\ref{prop:harmonic-lifts} are locally related by unitary isomorphisms of Higgs bundles. These isomorphisms preserve the determinant trivialization and have coefficients that depend smoothly on the stratum and surface variables.
\end{lem}

\begin{proof}
We use smooth trivializations of the families of curves on an overlap and pull both families back to one reference surface and bundle. The two lifts $c_1(u),c_2(u)$ are pointwise unitarily equivalent by Section~\ref{sec:topology}. On $\Hom(E,E)$, we define the differential in degree zero of the Hom complex by
\begin{equation*}
 D^{12}_{0,u}(g)=
 \bigl(\bar\partial_1\circ g-g\circ\bar\partial_2,\;
                   \varphi_1g-g\varphi_2\bigr).
\end{equation*}
An invertible $g\in\ker D^{12}_{0,u}$ is exactly an isomorphism of Higgs bundles satisfying $c_2(u)=c_1(u)^g$. The Laplacian $(D^{12}_{0,u})^*D^{12}_{0,u}$ is a smooth nonnegative elliptic family. At each parameter, composition with an isomorphism of Higgs bundles gives a linear isomorphism between $\ker D^{12}_{0,u}$ and the endomorphism algebra. Lemma~\ref{lem:constant-endomorphism-algebra} therefore gives $\dim_\C\ker D^{12}_{0,u}=\dim_\C\mathcal A_x$. The spectral argument in the proof of Proposition~\ref{prop:hodge} gives orthogonal projections $\mathsf H^{12}_{0,u}$ onto these kernels, depending smoothly on $u$ as operators on $W^{m,2}$ for every integer $m\geq0$.

At one parameter, we choose a unitary isomorphism $g_0$ preserving the determinant trivialization. The section $g(u)=\mathsf H^{12}_{0,u}g_0$ lies in $\ker D^{12}_{0,u}$ and remains invertible near that parameter by Sobolev embedding. For a fixed nearby parameter, let $g_{\mathrm{un}}$ be any unitary element of $\ker D^{12}_{0,u}$ and put $s_g=g_{\mathrm{un}}^{-1}g$. The endomorphism $s_g$ of the second harmonic Higgs bundle is parallel, and its adjoint is again a Higgs endomorphism by Section~\ref{sec:stabilizers}. Hence $s_g^*s_g$ is a positive parallel endomorphism and functional calculus takes place in the same algebra closed under the adjoint. The unitary factor in the polar decomposition
\begin{equation*}
 \widetilde g=g(g^*g)^{-1/2}
    =g_{\mathrm{un}}s_g(s_g^*s_g)^{-1/2}
\end{equation*}
satisfies
\begin{equation*}
 D^{12}_{0,u}(\widetilde g)=0,
 \qquad \widetilde g^*\widetilde g=\Id.
\end{equation*}
Its coefficients depend smoothly on $u$ and the surface variable.

Both determinant connections are trivial through $\tau$, so $\det\widetilde g\colon\Sigma\to S^1$ is parallel and therefore constant on the connected surface. After restricting to a smaller neighborhood of the chosen parameter, a smooth function $\zeta(u)\in S^1$ satisfies $\zeta(u)^r\det\widetilde g(u)=1$. The unitary gauge transformation $\zeta\widetilde g$ has determinant one and gives the required family of isomorphisms of Higgs bundles.
\end{proof}

We next represent tangent vectors to the fixed-curve strata by invariant harmonic cocycles orthogonal to the gauge orbits.

Let $u=(b,v)$ and let $\widehat c_x(u)$ be the solution from Proposition~\ref{prop:harmonic-lifts}. We denote by $\delta_{0,u}$ the infinitesimal action of the unitary gauge group of determinant one on the affine space of pairs $(\nabla,\varphi)$ over $X_b$, where $\nabla$ is unitary with fixed determinant and $\varphi$ is trace-free. Let $\boldsymbol\mu_{i,u}$, $1\leq i\leq3$, be the three Hitchin moment maps, and let $\mathbb I_{i,b}$ be the associated complex structures. Thus $(I_b,J_b,K_b)=(\mathbb I_{1,b},\mathbb I_{2,b}, \mathbb I_{3,b})$.

\begin{lem}\label{lem:fixed-fiber-harmonic-tangent}
At $\widehat c_x(u)$, the tangent space to the orbit-type stratum through $[\widehat c_x(u)]\in\cM_{\mathrm{Hit},b}$ is represented by
\begin{equation}\label{eq:fixed-fiber-horizontal}
 \mathcal H^{\mathrm{Hit}}_u=
 \left(
   \ker\delta_{0,u}^*\cap
   \bigcap_{i=1}^3\ker d\boldsymbol\mu_{i,u}
 \right)^{K_x}.
\end{equation}
These representatives are $L^2$-orthogonal to the gauge orbit. The map sending a connection variation to its $(0,1)$ part gives a canonical complex-linear isomorphism $\mathcal H^{\mathrm{Hit}}_u \simeq\mathcal H^1_{u,\mathrm{inv}}$. This space carries the metric and the three complex structures induced by the fixed-curve hyperk\"ahler quotient.
\end{lem}

\begin{proof}
In the notation of~\cite[Section~4.1]{FanOrbitTypes}, the operators $d_1$, the moment maps $\mu_i$, and the complex structures $I_i$ are, respectively,
\begin{equation*}
 d_1=\delta_{0,u},\qquad
 \mu_i=\boldsymbol\mu_{i,u},\qquad
 I_i=\mathbb I_{i,b}.
\end{equation*}
Writing $\star_b$ for the Hodge star on $X_b$, the moment-map identity is
\begin{equation}\label{eq:fan-moment-crosswalk}
 (d\boldsymbol\mu_{i,u})^*
   =-\mathbb I_{i,b}\,\delta_{0,u}\,\star_b.
\end{equation}
The three subspaces $\mathbb I_{i,b}\operatorname{im}\delta_{0,u}$ are orthogonal at a solution.

The orthogonal parallel splitting
\begin{equation*}
 \End E=\End_0E\oplus\C\Id
\end{equation*}
separates the fixed-determinant, trace-free affine space from the scalar summand. The infinitesimal action, its adjoint, the principal symbols, the three complex structures, and the linearized moment maps preserve this splitting, while the scalar equations decouple. The unitary gauge group of determinant one is a closed subgroup of the unitary gauge group and acts properly. Its Coulomb slice is constructed using the inverse of $\delta_{0,u}^*\delta_{0,u}$ on the $L^2$-orthogonal complement of $\ker\delta_{0,u}$ in the trace-free gauge algebra. Exponentiation stays inside the group of transformations of determinant one. The elliptic estimate on this complement gives closed range for $\delta_{0,u}$ and hence for each $\mathbb I_{i,b}\delta_{0,u}$.

In this slice, stabilizer containment together with equality of dimensions and numbers of components implies that the orbit-type locus through $\widehat c_x(u)$ equals the $K_x$-fixed locus. Together with~\eqref{eq:fan-moment-crosswalk} and the orthogonality of these three ranges, this allows the argument of~\cite[Section~4.1]{FanOrbitTypes} to apply to the trace-free summand. It gives a smooth orbit-type stratum whose tangent vectors have representatives in~\eqref{eq:fixed-fiber-horizontal}, together with its induced hyperk\"ahler tensors. Sending a connection variation to its $(0,1)$ part gives the stated complex-linear isomorphism with the invariant harmonic space in Proposition~\ref{prop:hodge}, as in~\cite[Proposition~2.2]{FanOrbitTypes}.
\end{proof}

\begin{proof}[Proof of Theorem~\ref{thm:main-harmonic}]
We work with the lifts of Proposition~\ref{prop:harmonic-lifts}. By Lemma~\ref{lem:fixed-fiber-harmonic-tangent}, tangent vectors to the orbit-type strata on a fixed curve have unique representatives in $\mathcal H^1_{u,\mathrm{inv}}$. We compare this bundle with the vertical tangent of the Kuranishi chart. For $w\in V_x^{H_x}$, let $\operatorname{KS}_u(w)$ be the vertical Kodaira--Spencer cocycle obtained by differentiating the coefficients of $\theta_x$ in the fixed smooth trivialization and viewing the result as $j_b$-type forms. We define
\begin{equation}\label{eq:harmonic-tangent}
 L_u(w)=\mathsf H_{1,u}
       \left(\Ad((g_x^+(u))^{-1})\operatorname{KS}_u(w)\right).
\end{equation}
For every $w$, the vector $L_u(w)$ is an invariant solution of the linearized Hitchin equations. Under the harmonic identification at $u_x$, the map $L_{u_x}$ is the identity on $V_x^{H_x}$ by Proposition~\ref{prop:harmonic-lifts}. By Proposition~\ref{prop:hodge}, $\rank\mathcal H^1_{\mathrm{inv}}=\dim_\C V_x^{H_x}$. Continuity of $L_u$ therefore implies that it remains an isomorphism after shrinking $B_x\times O_x$.

In the following differential formulas, a tangent vector is represented by the $(0,1)$ part of its connection component together with its Higgs-field component. The resulting pair lies in the degree-one term of the Higgs deformation complex at $u$, whose infinitesimal operator for the complex gauge action is $D_{0,u}$. Differentiating the right gauge action in these coordinates gives
\begin{equation*}
 d_v\widehat c_x|_u(w)
 =\Ad((g_x^+(u))^{-1})\operatorname{KS}_u(w)
   +D_{0,u}\bigl((g_x^+(u))^{-1}d_vg_x^+|_u(w)\bigr).
\end{equation*}
The harmonic projection annihilates the second term, which proves~\eqref{eq:harmonic-tangent}. For fixed $b$, the Kodaira--Spencer map $\operatorname{KS}_u$ is complex linear. Conjugation and $\mathsf H_{1,u}$ are complex linear on the Higgs deformation complex. Thus $L_u$ intertwines the vertical complex structure of $\mathscr S$ and $\mathbb I_{1,b}$.

For fixed $b$, the map
\begin{equation*}
 O_x\longrightarrow\cM_{\mathrm{Hit},b},\qquad
 v\longmapsto[\widehat c_x(b,v)],
\end{equation*}
viewed as a map to the corresponding orbit-type stratum, is smooth and has the invertible differential $L_u$. It is therefore a local biholomorphism. This map is $v\mapsto\mathrm{HK}_b^{-1}(\chi_x(b,q_x(v)))$; hence $\mathrm{HK}_b$ is biholomorphic on every connected orbit-type stratum. Proposition~\ref{prop:conjugacy} gives the correspondence between compact and complex stabilizer conjugacy classes, and Proposition~\ref{prop:types} compares the connected refinements in the quotient charts.

We next descend the tangent identifications and tensors. We write $c_2=c_1^g$ for a unitary gauge transformation $g$ preserving the determinant trivialization, as in Lemma~\ref{lem:unitary-overlaps}. Let $D_{0,c_i}$ and $\mathsf H_{1,c_i}$ denote the infinitesimal operator for the complex gauge action and the harmonic projection in degree one at $c_i$. For every vertical parameter tangent vector $\dot u$, differentiation of the right action gives
\begin{equation}\label{eq:overlap-differential}
 dc_2(\dot u)=\Ad(g^{-1})dc_1(\dot u)
        +D_{0,c_2}\bigl(g^{-1}dg(\dot u)\bigr).
\end{equation}
The second term in~\eqref{eq:overlap-differential} is an infinitesimal gauge direction. Unitary conjugation intertwines the Laplacians, so
\begin{equation*}
 \mathsf H_{1,c_2}\Ad(g^{-1})
   =\Ad(g^{-1})\mathsf H_{1,c_1},
 \qquad \mathsf H_{1,c_2}D_{0,c_2}=0.
\end{equation*}
Two choices of $g$ differ by an element of the compact stabilizer, which acts trivially on its invariant harmonic space. Hence the maps $L_u$ descend canonically. Trace pairings, the $L^2$ metric and $I_b,J_b,K_b$ are invariant under unitary gauge transformations, so they descend as smooth vertical tensors. This descent also respects the permutations of stable isotypic blocks with the same pair $(d_\alpha,m_\alpha)$ allowed by Corollary~\ref{cor:normalizer-blocks}.

Compatible surface metrics for a fixed $j_b$ are conformal. The one-form $L^2$ pairing and the harmonic conditions here are conformally invariant. Curve and bundle reparameterizations commute with the same constructions. The descended tensors are therefore independent of the background surface metrics and of the curve and bundle trivializations. Their hyperk\"ahler identities and integrability on a fixed fiber follow from Lemma~\ref{lem:fixed-fiber-harmonic-tangent}. For each $b$, $\mathrm{HK}_b$ restricts to an $I_b$-biholomorphism from every connected orbit-type stratum of $\cM_{\mathrm{Hit},b}$ onto the corresponding connected component of $\mathscr S\cap p^{-1}(b)$.
\end{proof}

The invariant deformation complex gives an explicit dimension formula. The ranks $d_\alpha$ refer to the distinct stable Higgs factors in Section~\ref{sec:stabilizers}; the multiplicities are $m_\alpha$.

\begin{cor}\label{cor:stratum-dimension}
Let $\mathscr S$ be a stratum containing an object with stable isotypic decomposition $\bigoplus_{\alpha=1}^s(E_\alpha,\varphi_\alpha)\otimes\C^{m_\alpha}$, where $d_\alpha=\rank E_\alpha$. Then
\begin{equation}\label{eq:stratum-dimension}
 \begin{aligned}
 \rank_{\C}T_{\mathscr S/\cT}
   &=2(g-1)\left(\sum_{\alpha=1}^s d_\alpha^2-1\right)+2(s-1),\\
 \dim_\C\mathscr S
   &=3g-3+\rank_{\C}T_{\mathscr S/\cT}.
 \end{aligned}
\end{equation}
\end{cor}

\begin{proof}
At a parameter $u$, the holomorphic coefficient bundle of the invariant complex is
\begin{equation*}
 \mathscr C_x=
 \left\{(a_\alpha)\in\bigoplus_{\alpha=1}^s\End E_{\alpha,u}:
       \sum_{\alpha=1}^s m_\alpha\tr(a_\alpha)=0\right\}.
\end{equation*}
Its rank is $\sum_\alpha d_\alpha^2-1$. Riemann--Roch gives the Euler characteristic of its Higgs complex as
\begin{equation*}
 \chi(\mathscr C_x)-\chi(\mathscr C_x\otimes K_{X_b})
 =-2(g-1)\left(\sum_\alpha d_\alpha^2-1\right).
\end{equation*}
Lemma~\ref{lem:hodge-ranks} gives $\dim_\C\mathcal H^0_{u,\mathrm{inv}}=\dim_\C\mathcal H^2_{u,\mathrm{inv}}=s-1$. Proposition~\ref{prop:fixed-factor} gives $\rank_{\C}T_{\mathscr S/\cT}=\dim_\C V_x^{H_x}=\dim_\C\mathcal H^1_{u,\mathrm{inv}}$. The Euler characteristic identity therefore proves the first formula. The second follows from the submersion property and $\dim_\C\cT=3g-3$.
\end{proof}

For $r=2$, a stable object has vertical stratum dimension $6g-6$. A strictly polystable object has the following form, with $L$ a degree-zero line bundle:
\begin{equation*}
 (E,\varphi)=(L,\eta)\oplus(L^{-1},-\eta),
 \qquad L\in\operatorname{Pic}^0(X_b),\quad\eta\in H^0(X_b,K_{X_b}).
\end{equation*}
If the two Higgs factors are nonisomorphic, their stabilizer of determinant one is $\C^*$ and the vertical dimension is $2g$. The factors are isomorphic precisely when $L^{\otimes2}\simeq\cO_{X_b}$ and $\eta=0$; the stabilizer is then $\SL_2(\C)$ and the vertical dimension is zero. Thus the three possible dimensions of strata of $\cM$ are $9g-9$, $5g-3$, and $3g-3$, respectively. The connected components of an orbit-type locus can give distinct strata of the same dimension.

\subsection{Gauge preimages of strata}
\label{sec:gauge-preimages}
Over a Beltrami ball $B\subset\cT$, let $\pi_{\C}$ be the quotient map from polystable $W^{k,2}$ Higgs pairs to $p^{-1}(B)$, and let $\pi_{\mathrm{Hit}}$ be the quotient map from $W^{k,2}$ solutions of the Hitchin equations. We use the relative complex configuration coordinates of~\cite[Proposition~2.4]{CaoUniversalPolystable} and the unitary configuration space of~\cite[Definition~2.5]{CaoUniversalPolystable}; their underlying real Banach manifolds are diffeomorphic via the Chern connection construction. The following extends the preimage and submersion assertions of~\cite[Theorems~A--B]{FanOrbitTypes} to the strata of $\cM$.

\begin{prop}\label{prop:gauge-preimages}
For every stratum $\mathscr S$, the inverse images of $\mathscr S\cap p^{-1}(B)$ under $\pi_{\C}$ and $\pi_{\mathrm{Hit}}$ are split embedded complex and real Banach submanifolds of their respective configuration spaces. The restricted quotient maps are holomorphic and smooth submersions, with tangent kernels equal to the corresponding infinitesimal gauge directions.

At $x\in\mathscr S\cap p^{-1}(B)$, after shrinking the chart of Proposition~\ref{prop:fixed-factor}, the maps
\begin{equation*}
\begin{aligned}
 F_{\C}\colon (B_x\times O_x)\times(H_x\backslash\cG_{k+1}^{\C})
   &\longrightarrow \pi_{\C}^{-1}(\mathscr S\cap\mathscr U_x),\\
 (u,H_xg)&\longmapsto\theta_x(u)^g,\\
 F_{\mathrm{Hit}}\colon (B_x\times O_x)\times(K_x\backslash\cG_{k+1})
   &\longrightarrow \pi_{\mathrm{Hit}}^{-1}(\mathscr S\cap\mathscr U_x),\\
 (u,K_xg)&\longmapsto\widehat c_x(u)^g
\end{aligned}
\end{equation*}
are a biholomorphism and a smooth diffeomorphism, respectively. In these coordinates, both quotient maps are the projection to $B_x\times O_x$ followed by $(b,v)\mapsto\chi_x(b,q_x(v))$.
\end{prop}

\begin{proof}
The local constructions can be centered at any $b_0\in B$ within the same Sobolev configuration model. To see this, let $\mu(b)$ be the Beltrami coefficient in a local coordinate $z$ for the original center, and put $\mu_0=\mu(b_0)$. A holomorphic coordinate $w$ for $j_{b_0}$ satisfies $dw=\lambda(dz+\mu_0\,d\bar z)$ for a nonvanishing smooth function $\lambda$. The coefficient of $j_b$ relative to $j_{b_0}$ is
\begin{equation*}
 \widetilde\mu(b)=\frac{\lambda}{\bar\lambda}\,
     \frac{\mu(b)-\mu_0}{1-\bar\mu_0\mu(b)}.
\end{equation*}
It is holomorphic in $b$, smooth on the surface, and zero at $b_0$. The changes of configuration coordinates use fixed spatial coordinate changes, smooth multiplication, and affine translations holomorphic in $b$, and are bounded at each Sobolev order. The Kuranishi reduction, quadratic obstruction identity, and gauge uniqueness in~\cite[Propositions~2.10 and~2.17 and Theorem~2.16]{CaoUniversalPolystable} use these properties and smallness of the coefficient; the harmonic projection cancellation uses only its type. The parameter-dependent real correction and its normalization along the zero section in~\cite[Proposition~3.1 and Lemma~3.2]{CaoUniversalPolystable} also apply. Thus the families in Section~\ref{sec:relative-families} and Proposition~\ref{prop:harmonic-lifts} are available at every center in this model.

At the chosen solution representing $x$, set
\begin{equation*}
 \mathfrak m_{\C}
   =\fh_x^{\perp_{L^2}}\cap W^{k+1,2}(\End_0E).
\end{equation*}
In the $W^{k,2}$ space of variations of Higgs pairs, the Hodge decomposition identifies $D_{0,x}\mathfrak m_{\C}$ and $\ker D_{0,x}^*$ as closed complementary subspaces. The map $D_{0,x}$ on $\mathfrak m_{\C}$ has bounded inverse $G_{0,x}D_{0,x}^*$ onto its image. In the relative affine coordinates centered at $x$, consider
\begin{equation*}
 \Psi(b,\eta,\xi)=(b,\eta)^{\exp\xi},
 \qquad D_{0,x}^*\eta=0,\quad \xi\in\mathfrak m_{\C}.
\end{equation*}
This map is holomorphic, and its differential at $(b_x,0,0)$ is
\begin{equation*}
 (\dot b,\dot\eta,\dot\xi)
   \longmapsto(\dot b,\dot\eta+D_{0,x}\dot\xi),
\end{equation*}
a bounded isomorphism. The inverse function theorem therefore gives an ambient holomorphic chart.

By the Kuranishi reduction, the integrable locus in this chart has Coulomb component $\eta=A_x(b,v)$ with $v\in Q_x$. Equation~\eqref{eq:stabilizer-comparison} and the argument of Proposition~\ref{prop:fixed-factor} show that a polystable pair in this chart has class in $\mathscr S$ exactly when $v\in O_x\subset V_x^{H_x}$. Hence the preimage of $\mathscr S$ in this chart is parametrized by
\begin{equation*}
 (b,v,\xi)\longmapsto
   \Psi\bigl(b,A_x(b,v),\xi\bigr),
 \qquad (b,v)\in B_x\times O_x.
\end{equation*}
The decomposition $\ker D_{0,x}^*=V_x\oplus(\im D_{1,x}^*\cap W^{k,2})$ and $A_x(b,v)=v+w_x(b,v)$ exhibit this subset as a holomorphic graph over $B_x\times O_x\times\mathfrak m_{\C}$ with a closed complement. It is therefore split embedded, and $\pi_{\C}$ is the projection to $(b,v)$; its differential has kernel $\im D_{0,c}$ at a pair $c$. Every polystable pair is the image of a smooth solution to the Hitchin equations under a Sobolev complex gauge transformation by Section~\ref{sec:topology}. Since each fixed gauge transformation acts by an ambient biholomorphism, these charts cover the inverse image.

We now regard this complex submanifold as a real Banach manifold and use the harmonic lift $\widehat c_x(u)$. Set
\begin{equation*}
 \mathfrak m
   =\fk_x^{\perp_{L^2}}\cap\Lie\cG_{k+1},
 \qquad \mathfrak m_{\C}=\mathfrak m\oplus i\mathfrak m.
\end{equation*}
For $s,\xi\in\mathfrak m$ near zero, the map
\begin{equation*}
 \Xi(u,s,\xi)=\widehat c_x(u)^{\exp(is)\exp\xi}
\end{equation*}
has differential $d\widehat c_x(\dot u)+D_{0,x}(is+\xi)$ at $(u_x,0,0)$. This is a bounded real isomorphism: the gauge summand is $\ker d\pi_{\C}$, and $d\widehat c_x$ is a right inverse of $d\pi_{\C}$ in the coordinates $u$. Thus $\Xi$ is a smooth chart after shrinking.

In this chart the solutions to the Hitchin equations are exactly $s=0$. One implication follows from unitary invariance. Conversely, if $\Xi(u,s,\xi)$ is a solution, so is $\widehat c_x(u)^{\exp(is)}$. Uniqueness up to the unitary gauge action gives a unitary gauge transformation $a$ with $\widehat c_x(u)^{\exp(is)}=\widehat c_x(u)^a$. The stabilizer in Proposition~\ref{prop:harmonic-lifts} then gives $\gamma=\exp(is)a^{-1}\in H_x$, and positivity implies
\begin{equation*}
 \exp(is)=(\gamma\gamma^*)^{1/2}\in H_x.
\end{equation*}
Indeed, the square root and its Hermitian logarithm belong to the algebra closed under the adjoint from Section~\ref{sec:stabilizers}; the positive endomorphism has determinant one, so its Hermitian logarithm is trace-free. Hence $is\in i\fk_x$, and $s\in\fk_x\cap\mathfrak m=\{0\}$. The subspace $s=0$ is split embedded, and its quotient map is projection to $u$, with differential having kernel $\im\delta_{0,c}$. Centering at each solution to the Hitchin equations and translating by fixed unitary gauge transformations proves the assertion throughout $\pi_{\mathrm{Hit}}^{-1}(\mathscr S\cap p^{-1}(B))$.

It remains to include every gauge transformation in the product charts. The finite-dimensional closed subgroups $H_x$ and $K_x$ are split Lie subgroups of the respective Banach gauge groups, so their spaces of left cosets are Banach manifolds with local quotient sections. By the stabilizer identity~\eqref{eq:stabilizer-comparison}, $F_{\C}$ is bijective. Its differential maps the coset factor isomorphically onto $\ker d\pi_{\C}$, while the differential of $\theta_x$ projects isomorphically onto the base. Thus $F_{\C}$ is a local biholomorphism at every point and, being bijective, a biholomorphism. Similarly, uniqueness up to the unitary gauge action and Proposition~\ref{prop:harmonic-lifts} imply that $F_{\mathrm{Hit}}$ is bijective. The charts $\Xi$ and their unitary translates show that it is a local diffeomorphism everywhere, and hence a diffeomorphism.

On an overlap in the configuration model, the inverse product maps and local quotient sections give holomorphic complex gauge transformations between the families $\theta_x$ and smooth unitary gauge transformations between the families $\widehat c_x$, the latter also supplied by Lemma~\ref{lem:unitary-overlaps}. If two sections satisfy $c_j=c_i^{a_{ij}}$, their stabilizers satisfy $H_j=a_{ij}^{-1}H_i a_{ij}$, and the fiber coordinate changes by $H_i g\mapsto H_j a_{ij}^{-1}g$. The same formula with $K_i,K_j$ holds for the harmonic sections. These transitions are holomorphic and smooth, respectively.
\end{proof}

\section{Relative symplectic forms}\label{sec:forms}
We first prove the trace identity for the relative Kuranishi families and then descend it to every stratum of Theorem~\ref{thm:main-strata}. On each local orbit-type stratum, the descended form agrees with the form obtained by symplectic reduction.

\subsection{The vertical trace pairing on Higgs pairs}
\label{sec:relative-coordinates}
In the chart centered at $x$, let $\mu(b)\in A^{0,1}(T_X)$ be the harmonic Beltrami differential representing $b$, with $\|\mu(b)\|_\infty<1$. For
\begin{equation*}
 \alpha=(a,\psi)\in W^{k,2}
   \bigl(A^{0,1}(\End_0E)\oplus A^{1,0}(\End_0E)\bigr),\qquad
 \varphi_{\mathrm{tot}}=\varphi+\psi,
\end{equation*}
the induced Dolbeault operator and Higgs field on $X_b$ are
\begin{equation*}
 \bar\partial_{b,\alpha}
   =\bar\partial_E-\iota_{\mu(b)}\partial_h+a,
 \qquad
 \varphi_{b,\alpha}
   =\varphi_{\mathrm{tot}}
      +\iota_{\mu(b)}\varphi_{\mathrm{tot}}.
\end{equation*}
Here $\iota_{\mu(b)}$ denotes contraction by $\mu(b)$, and $\partial_h$ is the Chern $(1,0)$ operator at $x$. If $\mu(b)=\mu_{\bar z}^{\,z}(b) d\bar z\otimes\partial_z$, then $\bar\partial_{b,\alpha}$ is evaluated on the vector $\partial_{\bar z}-\mu_{\bar z}^{\,z}(b)\partial_z$. The integrability equation and right gauge action are
\begin{equation*}
 F(b,\alpha)
   =\bar\partial_{\End}\varphi_{\mathrm{tot}}
     +\partial_h(\iota_{\mu(b)}\varphi_{\mathrm{tot}})
     +[a,\varphi_{\mathrm{tot}}]=0,
\end{equation*}
\begin{equation*}
 a^g=g^{-1}ag
     +g^{-1}(\bar\partial_{\End}g-\iota_{\mu(b)}\partial_hg),
 \qquad
 \varphi_{\mathrm{tot}}^g=g^{-1}\varphi_{\mathrm{tot}}g.
\end{equation*}
These formulas express the relative integrability equation in the fixed smooth trivialization used in~\cite[Proposition~2.4, equations~(2)--(4)]{CaoUniversalPolystable}.

For vertical variations of a Higgs pair on $X_b$, define
\begin{equation*}
 \Omega^{\mathrm{cfg}}_{/B_x}
 ((\dot{\bar\partial}_1,\dot\varphi_1),
  (\dot{\bar\partial}_2,\dot\varphi_2))
 =\int_{X_b}\tr(
       \dot{\bar\partial}_1\wedge\dot\varphi_2
      -\dot{\bar\partial}_2\wedge\dot\varphi_1).
\end{equation*}

\begin{lem}\label{lem:configuration-pairing}
In the coordinates above,
\begin{equation}\label{eq:configuration-pairing}
 \Omega^{\mathrm{cfg}}_{/B_x}
       ((\dot a_1,\dot\psi_1),(\dot a_2,\dot\psi_2))
 =\int_X\tr(\dot a_1\wedge\dot\psi_2-\dot a_2\wedge\dot\psi_1).
\end{equation}
This form is holomorphic and relatively closed. For fixed $b$, it is invariant under the complex gauge group of determinant one. On the integrable locus it annihilates an infinitesimal vertical gauge direction paired with any infinitesimal vertical integrable deformation. These assertions hold for smooth pairs and for pairs of class $W^{k,2}$.
\end{lem}

\begin{proof}
We first express the $j_b$-type variations in the fixed smooth coordinates. In a local coordinate, write $\dot a=a_{\bar z}\,d\bar z$ and $\dot\psi=\psi_z\,dz$. The unique $(0,1)$-form for $X_b$ taking value $1$ on $\partial_{\bar z}-\mu_{\bar z}^{\,z}(b)\partial_z$ is
\begin{equation*}
 \vartheta_b=
 \frac{d\bar z+\overline{\mu_{\bar z}^{\,z}(b)}\,dz}
      {1-|\mu_{\bar z}^{\,z}(b)|^2}.
\end{equation*}
The induced Dolbeault variation on $X_b$ is $a_{\bar z}\vartheta_b$, and the induced Higgs variation is $\psi_z(dz+\mu_{\bar z}^{\,z}(b)\,d\bar z)$. A direct calculation gives
\begin{equation*}
 \vartheta_b\wedge
 (dz+\mu_{\bar z}^{\,z}(b)\,d\bar z)=d\bar z\wedge dz.
\end{equation*}
The trace integrand is consequently $\tr(a_{1,\bar z}\psi_{2,z}-a_{2,\bar z}\psi_{1,z}) d\bar z\wedge dz$. This proves~\eqref{eq:configuration-pairing} for every $b$. Both sides are tensorial, so the coordinate identities agree on overlaps. The expression has constant coefficients in the fixed Banach coordinates and hence is holomorphic and relatively closed.

A fixed complex gauge transformation transforms both variations by conjugation. Cyclicity of the trace proves invariance. At an integrable Higgs pair $(\bar\partial_c,\varphi_c)$, a gauge direction for the right action is $(\bar\partial_c\xi,[\varphi_c,\xi])$. An integrable infinitesimal deformation satisfies $\bar\partial_c\psi+[a,\varphi_c]=0$. Integration by parts gives
\begin{equation}\label{eq:gauge-horizontal}
 \begin{aligned}
 \Omega^{\mathrm{cfg}}_{/B_x}
  ((\bar\partial_c\xi,[\varphi_c,\xi]),(a,\psi))
   &=-\int_{X_b}\tr
       \bigl(\xi(\bar\partial_c\psi+[a,\varphi_c])\bigr)\\
   &=0 .
 \end{aligned}
\end{equation}
The sign is that of the right action in Section~\ref{sec:standing}. Products and integrals are continuous for $k\geq4$. Smooth approximation therefore proves the same formula at the stated Sobolev orders.
\end{proof}

\begin{prop}\label{prop:trace-identity}
For every $(b,v)\in B_x\times U_x$ and $w_1,w_2\in V_x$, the Kuranishi map $A_x$ of Section~\ref{sec:relative-families} satisfies
\begin{equation}\label{eq:kuranishi-trace}
 \Omega^{\mathrm{cfg}}_{/B_x}
       (d_vA_x(b,v)w_1,d_vA_x(b,v)w_2)
       =\Omega_{x,\C}(w_1,w_2).
\end{equation}
The left-hand side uses the induced $(0,1)$- and $(1,0)$-forms on $X_b$. The resulting identity of relative two-forms is preserved by holomorphic pullback to any reduced analytic base mapping to $Z_x$.
\end{prop}

\begin{proof}
The fixed-curve cancellation in~\cite[Lemma~7.2]{FanOrbitTypes} uses the Hodge operators determined by the solution at $x$. We use
\begin{equation*}
 D''=\bar\partial_{\End}+[\varphi,\cdot],\qquad
 D'=\partial_h+[\varphi^{*h},\cdot].
\end{equation*}
The harmonic identities in~\cite[Section~2.1 and proof of Theorem~2.16]{CaoUniversalPolystable} give $(D')^2=0$, the graded derivation rule, and $D'u=0$ for $u\in V_x$. For smooth $w\in\im D_{1,x}^*$ they give
\begin{equation*}
 w=D'f,\qquad f=i\Lambda G_{2,x}D_{1,x}w .
\end{equation*}
Here $\Lambda$ is contraction with the K\"ahler form fixed at $x$. For every integer $m\geq1$, the operator $i\Lambda G_{2,x}D_{1,x}$ is bounded from $W^{m,2}$ to $W^{m+1,2}$ and commutes with $H_x$. All identities restrict to the parallel summand $\End_0E$.

For smooth endomorphism-valued one-forms, define
\begin{equation*}
 \langle s,t\rangle_{\tr}=\int_X\tr(s\wedge t).
\end{equation*}
For $s=a+\psi$, $t=a'+\psi'$, this is~\eqref{eq:trace}. If $f$ is a zero-form and $D't=0$, the derivation rule and cyclic trace give
\begin{equation*}
 \tr(D'f\wedge t)=\partial\tr(ft).
\end{equation*}
The $(1,0)$ part of $\tr(ft)$ has zero $\partial$ on a curve. For its $(0,1)$ part, $\partial$ is the exterior derivative. Stokes' theorem therefore gives
\begin{equation}\label{eq:trace-coboundary}
 \langle D'f,u\rangle_{\tr}=0,
 \qquad
 \langle D'f,D'g\rangle_{\tr}=0
\end{equation}
for every harmonic one-form $u$ and all smooth zero-forms $f,g$. Skew-symmetry gives the corresponding identity with $D'f$ in the second argument.

We define $f_x(b,v)=i\Lambda G_{2,x}D_{1,x}w_x(b,v)$. The regularity of $w_x$ and the boundedness of $i\Lambda G_{2,x}D_{1,x}$ imply that $f_x$ is holomorphic on $B_x\times U_x$ with values in $W^{m,2}$ for every integer $m\geq0$. Moreover, $w_x=D'f_x$, so
\begin{equation*}
 d_vA_x(b,v)u
   =u+D'\bigl(d_vf_x(b,v)u\bigr).
\end{equation*}
We set $h_i=d_vf_x(b,v)w_i$. Expanding the trace pairing gives
\begin{align*}
 &\langle w_1+D'h_1,w_2+D'h_2\rangle_{\tr}\\
 &\quad=\langle w_1,w_2\rangle_{\tr}
   +\langle D'h_1,w_2\rangle_{\tr}\\
 &\qquad\quad+\langle w_1,D'h_2\rangle_{\tr}
   +\langle D'h_1,D'h_2\rangle_{\tr}\\
 &\quad=\Omega_{x,\C}(w_1,w_2),
\end{align*}
where~\eqref{eq:trace-coboundary} eliminates the last three terms. By~\eqref{eq:configuration-pairing}, the left-hand side is the vertical trace pairing on $X_b$, which proves~\eqref{eq:kuranishi-trace}. Equivalently, on the smooth parameter space $B_x\times U_x$ one has
\begin{equation*}
 A_x^*\Omega^{\mathrm{cfg}}_{/B_x}
   =\pr_{V_x}^*\Omega_{x,\C}.
\end{equation*}
Functorial pullback gives this identity in $\Omega^2_{T/\cT}$ for every holomorphic map from a reduced analytic base $T$ to $Z_x$.
\end{proof}

\subsection{Descent to the strata}

\begin{prop}\label{prop:relative-descent}
Each stratum $\mathscr S$ has a canonical holomorphic relative form $\Omega^{\tr}_{\mathscr S/\cT}$, relatively closed and nondegenerate on $T_{\mathscr S/\cT}$. It is evaluated by the trace pairing on vertical deformation cocycles and agrees with the invariant harmonic pairing of Theorem~\ref{thm:main-harmonic}. For every holomorphic family of polystable objects over a reduced analytic germ whose classes lie in $\mathscr S$, the family trace pairing is the pullback of this form by its classifying map.
\end{prop}

\begin{proof}
On a chart parametrized by $B_y\times O_y$, with $O_y\subset V_y^{H_y}$, we define the form by taking the trace pairing of the vertical derivatives of $\theta_y$. By Proposition~\ref{prop:trace-identity}, this is the constant form $\Omega_{y,\C}|_{V_y^{H_y}}$. It is holomorphic and relatively closed, and the averaging argument in Section~\ref{sec:fixed-factors} proves nondegeneracy.

We first show that every family in the statement is locally induced from such a fixed-factor family. Let $T$ be its reduced parameter germ and $y$ the class over its distinguished point. Choose an identification of that fiber with a representative of $y$, preserving the determinant trivialization. Local completeness and regularity~\cite[Proposition~2.10, Lemmas~2.11--2.13, and Theorem~2.14]{CaoUniversalPolystable} give
\begin{equation*}
 T\xrightarrow{\ \sigma_T\ }Z_y
   =B_y\times(Q_y\cap U_y),
 \qquad
 \mathbf c_T\simeq\sigma_T^*\theta_y.
\end{equation*}
After shrinking $T$, the composite
\begin{equation*}
 T\xrightarrow{\sigma_T}Z_y
  \xrightarrow{\id_{B_y}\times q_y}B_y\times Y_y
\end{equation*}
takes values in $B_y\times q_y(O_y)$, which parametrizes a neighborhood of $y$ in $\mathscr S$. If $\sigma_T(t)=(b,w)$, then polystability implies that $H_yw$ is closed. Proposition~\ref{prop:fixed-factor} gives a unique $f\in O_y$ with $q_y(w)=q_y(f)$. The orbit $H_yf$ is the singleton $\{f\}$, so uniqueness of the closed orbit in a quotient fiber implies $w=f$. The linear equations defining $V_y^{H_y}$ vanish at every point of the reduced space $T$, hence as holomorphic functions. Thus $\sigma_T$ factors holomorphically through $B_y\times O_y$ and represents the classifying map in these coordinates.

The local forms agree on smooth overlaps. At a point of an overlap, choose a third chart centered at that point and apply the preceding factorization to the two chart families. Their maps to the fixed factor coincide because they have the same classifying map, so the two families are locally isomorphic. Differentiate this isomorphism in vertical directions, keeping the curve parameter fixed. The resulting gauge terms have zero trace pairing by~\eqref{eq:gauge-horizontal}. At a fixed curve parameter, a change of smooth curve trivialization is an orientation-preserving diffeomorphism and preserves the trace integral. The local forms therefore agree on all vertical tangent vectors. The strata are smooth and their projections to $\cT$ are submersions, so this is equality of holomorphic relative forms. They glue to $\Omega^{\tr}_{\mathscr S/\cT}$.

For the family assertion on a possibly singular base, work in one Teichm\"uller chart $B$ and choose a closed graph embedding $T\hookrightarrow P=B\times\Delta$ over $B$, where $\Delta$ is a polydisc. Let $\mathcal I_T$ be its defining ideal. We use the full sheaf of relative K\"ahler differentials and its exterior powers. The conormal presentation gives, for $q\geq1$,
\begin{equation*}
 \Omega^q_{T/B}\cong
 \left.
 \frac{\Omega^q_{P/B}}
 {\mathcal I_T\Omega^q_{P/B}
   +d_{P/B}\mathcal I_T\wedge\Omega^{q-1}_{P/B}}
 \right|_T .
\end{equation*}
In a fixed smooth trivialization, extend the family coefficient maps holomorphically to $P$, as in~\cite[Lemmas~2.12--2.13]{CaoUniversalPolystable}. Apply $d_{P/B}$ and the continuous trace-integral pairing, then take the image in $\Omega^2_{T/B}$. For scalar coefficients, two extensions differ by an element of $\mathcal I_T$, whose differential is zero in this quotient.

The same independence holds for the Sobolev-valued coefficients. Fix a sufficiently high Sobolev order, and let $\mathbf e$ be the difference of two coefficient extensions with values in the resulting Banach space $\mathcal E$. Every continuous scalar linear functional on $\mathcal E$ sends $\mathbf e$ into $\mathcal I_T$, because $T$ is reduced. At $t\in T$ and for each $N\geq1$, choose a continuous finite-rank projection $\mathsf P_N$ that is the identity on the span of the Taylor coefficients of $\mathbf e$ through degree $N$. The scalar coordinates of $\mathsf P_N\mathbf e$ lie in $\mathcal I_T$, so their relative differentials vanish in the conormal quotient. The remainder vanishes to order $N+1$ at $t$, and its derivative to order $N$. After contraction with the other coefficient derivatives and integration, the difference of the trace forms therefore lies in $\mathfrak n_t^N\Omega^2_{T/B,t}$, where $\mathfrak n_t$ is the maximal ideal of $\cO_{T,t}$. This holds for every $N$. Krull intersection for the finite $\cO_{T,t}$-module $\Omega^2_{T/B,t}$ proves that the difference is zero. The same finite-jet argument justifies the coefficient identities below after their continuous scalar trace contractions. The chain rule gives compatibility with holomorphic pullback; comparison in a common graph embedding proves independence of the embedding.

Write $d_T=d_{T/B}$. Suppose that two families in this chart satisfy $\mathbf c_2=\mathbf c_1^g$ for a holomorphic family isomorphism $g$ preserving the determinant trivialization. Differentiating the right gauge action gives
\begin{equation}\label{eq:family-gauge-derivative}
 d_T\mathbf c_2
 =\Ad(g^{-1})d_T\mathbf c_1
  +D_{0,\mathbf c_2}\bigl(g^{-1}d_Tg\bigr).
\end{equation}
Here $D_{0,\mathbf c_2}$ is the infinitesimal operator for complex gauge transformations at $\mathbf c_2$, acting on the curve coefficients, and $\tr(g^{-1}d_Tg)=0$ by the determinant condition. All parameter differentials are taken in the conormal quotient; neither integrability nor the gauge identities are required away from $T$. Differentiating integrability gives $D_{1,\mathbf c_i}d_T\mathbf c_i=0$, and integrability also gives $D_{1,\mathbf c_i}D_{0,\mathbf c_i}=0$. The integration-by-parts identity~\eqref{eq:gauge-horizontal} extends to these relative differential coefficients. It eliminates the mixed terms and the term pairing two gauge directions in $\Omega^2_{T/B}$, while cyclicity of the trace identifies the remaining terms. The family trace pairing is therefore invariant under these isomorphisms as a section of the full sheaf $\Omega^2_{T/B}$.

Apply this invariance to $\mathbf c_T\simeq\sigma_T^*\theta_y$. Proposition~\ref{prop:trace-identity} identifies the family trace pairing with the pullback of $\Omega^{\tr}_{\mathscr S/\cT}$ by the classifying map, as an equality in $\Omega^2_{T/\cT}$. Repeating the calculation in another Teichm\"uller chart yields the pullback of the same already-glued form. This proves independence of the chosen smooth curve trivialization on a singular base, as well as of the quotient chart and representative.

Finally, Theorem~\ref{thm:main-harmonic} provides smooth local harmonic lifts and invariant harmonic representatives of vertical tangent vectors. Differentiation of a gauge transformation and Hodge projection each change a deformation cocycle by a coboundary. By~\eqref{eq:gauge-horizontal}, the descended form therefore agrees with the invariant harmonic trace pairing and hence with the fiberwise holomorphic symplectic form normalized by~\eqref{eq:trace}.
\end{proof}

\begin{prop}\label{prop:chart-forms}
The quotient charts can be chosen so that, for every connected orbit-type stratum $\mathscr R\subset\mathcal W_x$,
\begin{equation}\label{eq:chart-form}
 \left.\chi_x^*\Omega^{\tr}_{\mathscr S/\cT}
 \right|_{B_x\times\mathscr R}
 =\pr_{\mathscr R}^*\Omega_{\mathscr R}^{\tr}.
\end{equation}
Here $\mathscr S$ is the unique global stratum containing $\chi_x(B_x\times\mathscr R)$, and $\Omega_{\mathscr R}^{\tr}$ is the reduction of $\Omega_{x,\C}$ on $Q_x$. These identities are compatible with chart overlaps.
\end{prop}

\begin{proof}
The theorem on linear hyperk\"ahler quotients~\cite[Section~4.2, Theorem~1.4(i)--(ii), and Lemma~4.11]{Mayrand2020Stratification} applies to $(H_x,V_x,\Omega_{x,\C},\nu_{x,\C})$. Indeed, the compact group admits the required holomorphic complexification, the action is algebraic and symplectic, and~\eqref{eq:moment} gives the quadratic normalization. A $K_x$-invariant linear hyperk\"ahler structure compatible with $\Omega_{x,\C}$ supplies the required reduction. Its reduced complex analytic quotient is $Y_x$, by the quotient comparison in Section~\ref{sec:local}.

For a connected orbit-type stratum $\mathscr R\subset\mathcal W_x$, define
\begin{equation*}
 \widetilde{\mathscr R}
 =\{v\in q_x^{-1}(\mathscr R):H_xv\text{ is closed}\}.
\end{equation*}
The cited lemma states that $\widetilde{\mathscr R}$ is a complex submanifold and that $q_x\colon\widetilde{\mathscr R}\to\mathscr R$ is a surjective holomorphic submersion satisfying
\begin{equation*}
 q_x^*\Omega_{\mathscr R}^{\tr}
 =\Omega_{x,\C}|_{\widetilde{\mathscr R}}.
\end{equation*}
By Proposition~\ref{prop:types}, the products $B_x\times\mathscr R$ map into connected gauge orbit-type strata.

For the extended family,~\eqref{eq:relative-equivariance} gives
\begin{equation}\label{eq:saturated-family-equivariance}
 \widetilde\theta_x(b,h^{-1}\!\cdot v)
     =\widetilde\theta_x(b,v)^h.
\end{equation}
By~\eqref{eq:saturated-domain}, every $v_0\in\widetilde U_x$ has a neighborhood $N$ in $Q_x$ and an $h\in H_x$ such that $h\cdot N\subset Q_x\cap U_x$. The extended family on $N$ is obtained from the original family on $h\cdot N$ by this fixed gauge transformation, as in~\cite[Lemma~2.20]{CaoUniversalPolystable}. The $H_x$-action preserves $\Omega_{x,\C}$. Differentiating~\eqref{eq:saturated-family-equivariance} and using cyclicity of trace therefore transports the identity of Proposition~\ref{prop:trace-identity} to $N$. This proves~\eqref{eq:kuranishi-trace} for the extended family throughout $\widetilde U_x$, also after pullback from a complex manifold.

The family on $B_x\times\widetilde{\mathscr R}$ is polystable, with classifying map $\chi_x\circ(\id_{B_x}\times q_x)$. Proposition~\ref{prop:relative-descent} and the extended trace identity give
\begin{equation*}
 (\id_{B_x}\times q_x)^*
       \chi_x^*\Omega^{\tr}_{\mathscr S/\cT}
 =\widetilde\theta_x^*\Omega^{\mathrm{cfg}}_{/B_x}
 =\pr_{\widetilde{\mathscr R}}^*\Omega_{x,\C}.
\end{equation*}
Together with $q_x^*\Omega_{\mathscr R}^{\tr} =\Omega_{x,\C}|_{\widetilde{\mathscr R}}$, this proves~\eqref{eq:chart-form}, because pullback under a surjective holomorphic submersion is injective on relative forms.
\end{proof}

\begin{proof}[Proof of Theorem~\ref{thm:main-forms}]
Proposition~\ref{prop:relative-descent} constructs the canonical relative forms and proves that they are closed, nondegenerate, and equal to the deformation and harmonic trace pairings. Proposition~\ref{prop:chart-forms} gives their comparison with the quadratic reduction forms on every local orbit-type stratum and for every base parameter.
\end{proof}

\section{Poisson structure}\label{sec:poisson}
Proposition~\ref{prop:chart-forms} implies that the local brackets agree on overlaps. We use this compatibility to construct a Poisson bracket on the reduced analytic structure sheaf, determine its symplectic leaves, and compute its behavior under changes of marking and scaling of the Higgs field.

The reduced analytic quotient $Y_x$ carries a holomorphic Poisson bracket whose restriction to every orbit-type stratum is the inverse of the reduced symplectic form~\cite[Theorem~1.4(ii), Section~4.2, and Lemmas~4.11--4.13]{Mayrand2020Stratification}. The proof of Proposition~\ref{prop:chart-forms} verifies the hypotheses and normalization of this result. Our convention is $\iota_{X_f}\Omega=df$ and $\{f,g\}=\Omega(X_f,X_g)$ for a symplectic form $\Omega$.

To construct the bracket on arbitrary open subsets, let $d=\dim B_x$ and adjoin a trivial symplectic summand with coordinates
\begin{equation*}
 (b_1,\ldots,b_d,\zeta_1,\ldots,\zeta_d).
\end{equation*}
The same quotient theorem applies to
\begin{equation*}
 V_x\oplus\C^{2d},\qquad
 \Omega_{x,\C}+\sum_{j=1}^d db_j\wedge d\zeta_j,
\end{equation*}
with $H_x$ acting trivially on the new summand. The quotient is $Y_x\times\C^{2d}$, and each orbit-type stratum is a product $\mathscr R\times\C^{2d}$ with the split reduced form.

\begin{prop}\label{prop:local-poisson}
The reduced analytic structure sheaf of $B_x\times Y_x$ carries a vertical holomorphic Poisson bracket on arbitrary open subsets. It is compatible with restriction and bilinear over holomorphic functions on $B_x$. On $B_x\times\mathscr R$, for every connected orbit-type stratum $\mathscr R\subset Y_x$, it is the inverse of $\Omega_{\mathscr R}^{\tr}$ in the $\mathscr R$-directions.
\end{prop}

\begin{proof}
Let $\mathscr U$ be any open subset of $B_x\times Y_x$. We pull $F,G\in\cO(\mathscr U)$ back to functions independent of $\zeta$ on $\mathscr U\times\C_\zeta^d$. On a product stratum, the product Poisson bracket is
\begin{align*}
 \{F,G\}_{\mathrm{prod}}
 &=\{F,G\}_{Y_x}
   +\sum_{j=1}^d
     \left(
       \frac{\partial F}{\partial b_j}
       \frac{\partial G}{\partial\zeta_j}
       -\frac{\partial F}{\partial\zeta_j}
       \frac{\partial G}{\partial b_j}
     \right)\\
 &=\{F,G\}_{Y_x}.
\end{align*}
Thus the product bracket of the pulled-back functions is pointwise independent of $\zeta$. Reducedness implies that it is the pullback of its restriction to $\zeta=0$. This restriction defines a holomorphic function on $\mathscr U$. The construction is compatible with restriction of open sets. The Jacobi identity, Leibniz rule, skew-symmetry, and bilinearity descend from the product bracket. The displayed formula also shows that base functions are Casimirs and that the construction is independent of the chosen auxiliary coordinates.
\end{proof}

\begin{proof}[Proof of Theorem~\ref{thm:main-poisson}]
We transport the local brackets of Proposition~\ref{prop:local-poisson} through the quotient charts. On an overlap, equation~\eqref{eq:chart-form} shows that both brackets restrict on every stratum to the inverse of the same relative form. A holomorphic function on a reduced analytic space that vanishes at every point is zero, so the two local brackets agree as sections. They therefore glue, and the same argument proves uniqueness.

The local construction gives bilinearity over $p^{-1}\cO_{\cT}$, vanishing brackets with pullbacks of base functions, and the Jacobi and Leibniz identities. Thus $\chi_x$ is a Poisson isomorphism from $B_x\times\mathcal W_x$, with the restricted vertical bracket induced by quadratic reduction, onto $\mathscr U_x$.
\end{proof}

\begin{lem}\label{lem:derivation-flows}
Let $X$ be a reduced complex analytic space with a locally finite partition into smooth locally closed submanifolds. Let $V$ be a derivation of $\cO_X$ whose restriction to each submanifold is a tangent holomorphic vector field. The local flow of $V$ preserves every connected member of the partition.
\end{lem}

\begin{proof}
We choose a local embedding of $X$ in $\C^N$ and let $z_1,\ldots,z_N$ be the ambient coordinates. Holomorphic lifts of the functions $V(z_i)|_X$ define
\begin{equation*}
 \widetilde V=\sum_{i=1}^N\widetilde{V(z_i)}
                 \frac{\partial}{\partial z_i}.
\end{equation*}
If $I_X$ is the defining ideal, then $\widetilde V(I_X)\subset I_X$. For local generators $f_1,\ldots,f_{N_I}$ of $I_X$, there are holomorphic functions $a_{ij}$ such that
\begin{equation*}
 \widetilde V(f_i)=\sum_{j=1}^{N_I} a_{ij}f_j.
\end{equation*}
Along an integral curve $\gamma$ of $\widetilde V$, the functions $f_i(\gamma(t))$ satisfy the homogeneous system
\begin{equation*}
 \frac{d}{dt}f_i(\gamma(t))
   =\sum_{j=1}^{N_I} a_{ij}(\gamma(t))f_j(\gamma(t)).
\end{equation*}
By uniqueness, these functions remain zero when $\gamma(0)\in X$; hence the ambient flow preserves $X$. On a smooth member of the partition, the restricted tangent field has its own local flow, which agrees with the ambient flow by uniqueness. The same argument at every time shows that the set of times spent in each member is open. These disjoint open sets partition the connected time domain, so the trajectory stays in a single member. Its continuous image stays in one connected component.
\end{proof}

\begin{cor}\label{cor:leaves}
The symplectic leaves of the bracket in Theorem~\ref{thm:main-poisson} are the connected components of $\mathscr S\cap p^{-1}(b)$, where $\mathscr S$ is a stratum and $b\in\cT$.
\end{cor}

\begin{proof}
Let $L$ be a connected component of $\mathscr S\cap p^{-1}(b)$. The bracket restricts to the inverse of $\Omega^{\tr}_{\mathscr S/\cT}$ on $L$ and is nondegenerate. In a local embedding of $\cM$ into $\C^N$, the restrictions of the ambient coordinate differentials span $T^*L$. Their Hamiltonian vector fields therefore span $TL$. Base functions are Casimirs, so all Hamiltonian flows remain in the fiber over $b$. Proposition~\ref{prop:local-poisson} shows that every Hamiltonian derivation restricts to a tangent holomorphic vector field on each orbit-type stratum, and Lemma~\ref{lem:derivation-flows} shows that its flow preserves every connected stratum.

For $z\in L$, let $V_1,\ldots,V_n$ be Hamiltonian vector fields whose values form a basis of $T_zL$, and let $\Phi_i^t$ denote their local flows. The map
\begin{equation*}
 F(t_1,\ldots,t_n)
   =\Phi_n^{t_n}\circ\cdots\circ\Phi_1^{t_1}(z)
\end{equation*}
satisfies $dF_0(e_i)=V_i(z)$, so its differential at the origin is invertible. The holomorphic inverse function theorem shows that the local Hamiltonian orbit of $z$ is open in $L$. These orbits form a disjoint open cover of the connected manifold $L$, so $L$ is a single Hamiltonian orbit and hence a maximal connected symplectic leaf.
\end{proof}

Let $\Diff^+(\Sigma)$ be the group of smooth diffeomorphisms of $\Sigma$ that preserve orientation. Its component group is $\Gamma_g=\pi_0(\Diff^+(\Sigma))$. We write $A_\gamma$, $\gamma\in\Gamma_g$, for the holomorphic action by changes of marking and $\rho_\lambda$, $\lambda\in\C^*$, for the holomorphic action by scaling the Higgs field. The action result~\cite[Corollary~5.1]{CaoUniversalPolystable} gives commuting holomorphic actions on the analytic space. Scaling is jointly holomorphic in $\lambda$ and fixes $\cT$, while the action by changes of marking covers the standard action on Teichm\"uller space.

\begin{prop}\label{prop:actions}
The action by changes of marking preserves the relative forms and the bracket of Theorem~\ref{thm:main-poisson}. For every $\lambda\in\C^*$ and every stratum $\mathscr S$,
\begin{equation*}
 \rho_\lambda^*\Omega^{\tr}_{\rho_\lambda(\mathscr S)/\cT}
     =\lambda\Omega^{\tr}_{\mathscr S/\cT},\qquad
 \rho_\lambda^*
   \Pi_{\rho_\lambda(\mathscr S)/\cT}
     =\lambda^{-1}\Pi_{\mathscr S/\cT},
\end{equation*}
where $\Pi_{\mathscr S/\cT}=(\Omega^{\tr}_{\mathscr S/\cT})^{-1}$. On the reduced structure sheaf,
\begin{equation*}
 \{\rho_\lambda^*f,\rho_\lambda^*g\}
     =\lambda\,\rho_\lambda^*\{f,g\}.
\end{equation*}
The two actions commute.
\end{prop}

\begin{proof}
Both actions preserve the determinant trivialization and trace-free condition. Scaling the Higgs field leaves the automorphism equations unchanged; a change of marking transports the automorphism group. Hence both actions permute the connected orbit-type strata.

Pullback by an orientation-preserving diffeomorphism of the curve and identification of the bundles preserve the trace integral. Consequently, $A_\gamma^*\Omega^{\tr}_{A_\gamma(\mathscr S)/\cT} =\Omega^{\tr}_{\mathscr S/\cT}$. On a fixed curve, scaling sends a deformation cocycle $(a,\psi)$ to $(a,\lambda\psi)$. Substitution in~\eqref{eq:trace} gives the first scaling identity. Inverting the nondegenerate vertical forms gives the stated tensor weight and bracket transformation on each stratum. The action by changes of marking preserves the bracket on every stratum. Reducedness extends these identities to the reduced structure sheaf $\cO_{\cM}$.
\end{proof}

\begin{rem}\label{rem:rank-one}
For rank one, $\tau\colon E\xrightarrow{\sim}\cO_X$ and the trace-free condition gives $\varphi=0$. Thus $\cM=\cT$ and $p$ is the identity. There is one stratum; its relative tangent bundle and relative two-form vanish. The Poisson bracket is zero and its leaves are points.
\end{rem}


\begin{thebibliography}{ACGFT26}
\interlinepenalty=10000

\bibitem[ACGFT26]{UniversalHitchinModuliSpaces}
Luis \'Alvarez-C\'onsul, Mario Garcia-Fernandez, Oscar Garc\'ia-Prada, and Samuel Trautwein, \emph{Universal Hitchin moduli spaces}, arXiv:2512.07553v2, 8 January 2026. \url{https://arxiv.org/abs/2512.07553v2}.

\bibitem[Cao26]{CaoUniversalPolystable}
Shiyu Cao, \emph{An analytic construction of the universal moduli space of $\SL_r(\C)$-Higgs bundles}, arXiv:2609.25807v3. \url{https://arxiv.org/abs/2609.25807v3}.

\bibitem[CTW25]{CTW25}
Brian Collier, J\'er\'emy Toulisse, and Richard Wentworth, \emph{Higgs Bundles, Isomonodromic Leaves and Minimal Surfaces}, preprint (2025), arXiv:2512.07152v1. \url{https://arxiv.org/abs/2512.07152v1}.

\bibitem[Fan22a]{FanQuasiprojectivity}
Yue Fan, \emph{An analytic approach to the quasi-projectivity of the moduli space of Higgs bundles}, Adv. Math. \textbf{406} (2022), Article 108506. \url{https://doi.org/10.1016/j.aim.2022.108506}. Preprint version: arXiv:2009.14680v1 (2020), \url{https://arxiv.org/abs/2009.14680v1}.

\bibitem[Fan22b]{FanConstruction}
Yue Fan, \emph{Construction of the moduli space of Higgs bundles using analytic methods}, Math. Res. Lett. \textbf{29} (2022), no.~4, 1011--1048. \url{https://doi.org/10.4310/MRL.2022.v29.n4.a5}. Preprint version: arXiv:2004.07182v2 (2020), \url{https://arxiv.org/abs/2004.07182v2}.

\bibitem[Fan22c]{FanOrbitTypes}
Yue Fan, \emph{The orbit type stratification of the moduli space of Higgs bundles}, Math. Ann. \textbf{383} (2022), 1033--1059. \url{https://doi.org/10.1007/s00208-021-02244-0}. Preprint version: arXiv:2005.13649v1 (2020). \url{https://arxiv.org/abs/2005.13649v1}.

\bibitem[Hit87]{Hitchin1987}
N.~J.~Hitchin, \emph{The self-duality equations on a Riemann surface}, Proc. London Math. Soc. (3) \textbf{55} (1987), no.~1, 59--126. \url{https://doi.org/10.1112/plms/s3-55.1.59}.

\bibitem[Hit26]{HitchinUniversalHiggsModuliSpace}
Nigel Hitchin, \emph{A universal Higgs bundle moduli space}, Journal of Geometry and Physics \textbf{229} (2026), Article 105955. \url{https://doi.org/10.1016/j.geomphys.2026.105955}. arXiv:2603.21712v2, 21 July 2026. \url{https://arxiv.org/abs/2603.21712v2}.

\bibitem[Kat95]{KatoPerturbation}
Tosio Kato, \emph{Perturbation Theory for Linear Operators}, Classics in Mathematics, Springer-Verlag, Berlin--Heidelberg, 1995, reprint of the 1980 corrected printing of the second edition. \url{https://doi.org/10.1007/978-3-642-66282-9}.

\bibitem[May22]{Mayrand2020Stratification}
Maxence Mayrand, \emph{Stratification of singular hyperk\"ahler quotients}, Complex Manifolds \textbf{9} (2022), no.~1, 261--284. \url{https://doi.org/10.1515/coma-2021-0140}. Preprint version: arXiv:1807.05992v2 (2020). \url{https://arxiv.org/abs/1807.05992v2}.

\bibitem[Mat12]{Mather2012}
John~N.~Mather, \emph{Notes on topological stability}, Bull. Amer. Math. Soc. (N.S.) \textbf{49} (2012), no.~4, 475--506. \url{https://doi.org/10.1090/S0273-0979-2012-01383-6}.

\bibitem[Nit91]{Nitsure1991}
Nitin Nitsure, \emph{Moduli space of semistable pairs on a curve}, Proc. London Math. Soc. (3) \textbf{62} (1991), no.~2, 275--300. \url{https://doi.org/10.1112/plms/s3-62.2.275}.

\bibitem[Sim92]{Simpson1992}
Carlos~T.~Simpson, \emph{Higgs bundles and local systems}, Publ. Math. IH\'ES \textbf{75} (1992), 5--95. \url{https://doi.org/10.1007/BF02699491}.

\bibitem[Sim94a]{Simpson1994I}
Carlos~T.~Simpson, \emph{Moduli of representations of the fundamental group of a smooth projective variety I}, Publ. Math. IH\'ES \textbf{79} (1994), 47--129. \url{https://doi.org/10.1007/BF02698887}.

\end{thebibliography}
\end{document}